\documentclass[a4paper,11pt]{amsart}
\usepackage{amsfonts,amsthm,amsmath,amssymb,stmaryrd}
\usepackage{enumitem}
\usepackage{mathtools}
\usepackage{tikz}
\usepackage{tikz-cd}
\usepackage{graphicx}
\usepackage{color}
\usepackage{dsfont}
\usepackage{adjustbox}
\usepackage{amsaddr} 

\usepackage[
textwidth=3cm,
textsize=small,
colorinlistoftodos]
{todonotes} 

\definecolor{mycitecolor}{rgb}{0,0,0.8}

\usepackage[bookmarks=true]{hyperref}
\hypersetup{colorlinks, linkcolor=black, citecolor=mycitecolor} 

\makeatletter
\pgfqkeys{/tikz/commutative diagrams}{
  row sep/.code={\tikzcd@sep{row}{#1}{}},
  column sep/.code={\tikzcd@sep{column}{#1}{}},
  bo row sep/.code={\tikzcd@sep{row}{#1}{between origins}},
  bo column sep/.code={\tikzcd@sep{column}{#1}{between origins}},
  bo column sep/.default=normal,
  bo row sep/.default=normal,
}
\def\tikzcd@sep#1#2#3{
  \pgfkeysifdefined{/tikz/commutative diagrams/#1 sep/#2}%
    {\pgfkeysalso{/tikz/#1 sep={\ifx\\#3\\1*\else1.7*\fi\pgfkeysvalueof{/tikz/commutative diagrams/#1 sep/#2},#3}}}%
    {\pgfkeysalso{/tikz/#1 sep={#2,#3}}}}
\makeatother

\newcommand{\btk}{\begin{tikzcd}}
\newcommand{\etk}{\end{tikzcd}}
\newcommand{\bc}{\begin{center}}
\newcommand{\ec}{\end{center}}

\newcommand{\Ch}{\operatorname{\mathsf{Ch^b}}}
\newcommand{\Chdw}{\operatorname{\mathsf{Ch^b_{dw}}}}

\newcommand{\loc}{\operatorname{\mathrm{loc}}}
\newcommand{\we}{\xrightarrow{\sim}}

\newcommand{\inj}{\operatorname{\mathsf{inj}}}
\newcommand{\Rinj}{R\mathsf{-inj}}
\newcommand{\proj}{\operatorname{\mathsf{proj}}}
\newcommand{\Rproj}{R\mathsf{-proj}}
\newcommand{\contr}{\operatorname{\mathsf{contr}}}
\newcommand{\CM}{\operatorname{\mathsf{CM}}}
\newcommand{\Rmod}{R\mathsf{-mod}}

\newcommand{\Sp}{\mathrm{Sp}}
\newcommand{\cof}{\hookrightarrow}

\DeclareMathOperator{\Hom}{Hom}

\DeclareMathOperator{\id}{id}

\DeclareMathOperator{\depth}{depth}

\DeclareMathOperator{\Span}{Span}
\DeclareMathOperator{\Ar}{Ar}

\DeclareMathOperator{\Ob}{Ob}

\DeclareMathOperator{\Ext}{Ext}
\DeclareMathOperator{\weq}{we}
\DeclareMathOperator{\co}{co}
\DeclareMathOperator{\cok}{coker}

\newcommand{\C}{\mathcal{C}}
\newcommand{\Z}{\mathcal{Z}}
\newcommand{\E}{\mathcal{E}}

\newcommand{\D}{\mathcal{D}}
\newcommand{\I}{\mathcal{I}}
\renewcommand{\P}{\mathcal{P}}

\newcommand{\F}{\mathcal{F}}

\newcommand{\A}{\mathcal{A}}
\newcommand{\B}{\mathcal{B}}

\theoremstyle{theorem}
\newtheorem{thm}{Theorem}[section]
\newtheorem{prop}[thm]{Proposition}
\newtheorem{lemma}[thm]{Lemma}

\theoremstyle{definition}
\newtheorem{defn}[thm]{Definition}

\theoremstyle{remark}
\newtheorem{ex}[thm]{Example}
\newtheorem{rmk}[thm]{Remark}

\newcommand{\stopnumberingbysection}{\renewcommand{\theequation}{\arabic{equation}}}

\theoremstyle{theorem}
\newtheorem{theorem}[equation]{Theorem}

\begin{document}

\title{Cotorsion pairs and a K-theory localization theorem}

\author{\scriptsize MARU SARAZOLA}
\address{\small Department of Mathematics, \\ Cornell University,\\
Ithaca NY, 14853, USA }
\email{mes462@cornell.edu}



\maketitle
\begin{abstract}
We show that a complete hereditary cotorsion pair $(\C,\C^\bot)$ in an exact category $\E$, together with a subcategory $\Z\subseteq\E$ containing $\C^\bot$, determines a Waldhausen category structure on the exact category $\C$, in which $\Z$ is the class of acyclic objects.

This allows us to prove a new version of Quillen's Localization Theorem, relating the $K$-theory of exact categories $\A\subseteq\B$ to that of a cofiber. The novel idea in our approach is that, instead of looking for an exact quotient category that serves as the cofiber, we produce a Waldhausen category, constructed through a cotorsion pair.
Notably, we do not require $\A$ to be a Serre subcategory, which produces new examples.

Due to the algebraic nature of our Waldhausen categories, we are able to
recover a version of Quillen's Resolution Theorem, now in a more homotopical setting that allows for weak equivalences.
\end{abstract}
\vspace{0.5cm}



\tableofcontents

\section{Introduction}

\stopnumberingbysection


Ever since its origin, algebraic $K$-theory has proved to be exceedingly hard to compute, which is why results relating the $K$-theory groups of different categories are of vital importance in the field. Notably, a major tool in this direction consists of finding homotopy fibration sequences $$K(\mathcal{X})\to K(\mathcal{Y})\to K(\mathcal{Z})$$ relating the $K$-theory spaces of categories $\mathcal{X}, \mathcal{Y}$ and $\mathcal{Z}$ because, since the $n$th $K$-theory group $K_n(\mathcal{X})$ is defined as the $n$th homotopy group of the $K$-theory space $K(\mathcal{X})$, such homotopy fibration sequences induce a long exact sequence $$\dots\to K_{n+1}(\mathcal{Z})\to K_n(\mathcal{X})\to K_n(\mathcal{Y})\to K_n(\mathcal{Z})\to K_{n-1}(\mathcal{X})\to\dots$$ ending in $K_0(\mathcal{X})\to K_0(\mathcal{Y})\to K_0(\mathcal{Z})\to 0$.

Among these results, one of the most useful
is Quillen's Localization Theorem, relating the $K$-theory of two abelian categories $\A\subseteq\B$, and that of their quotient $\B/\A$.

More precisely, let $\B$ be an abelian category, and $\A\subseteq\B$ a Serre subcategory; that is, an abelian subcategory closed under extensions, subobjects and quotients. In this situation, it is possible to define a quotient abelian category $\B/\A$, together with an exact quotient functor $\loc\colon\B\to\B/\A$, see \cite[Thm.\ 2.1]{Swan}, or also \cite[Ch.\ III 1,2]{Gabriel}. This quotient functor satisfies a universal property: given an exact functor $F\colon \B\to\C$ between abelian categories such that $F(A)\cong 0$ for every $A\in\A$, there exists a unique exact functor $F'\colon\B/\A\to\C$ with $F=F'\loc$.

Then, the algebraic $K$-theory of the categories $\A, \B$ and $\B/\A$ are related in the following way:
\begin{theorem}\cite[Thm.\ 5]{Qui73}
Let $\A$ be a Serre subcategory of a small abelian category $\B$. Then $$K(\A)\to K(\B)\xrightarrow{\loc} K(\B/\A)$$ is a homotopy fibration sequence.
\end{theorem}

Although immensely useful, this result suffers from an evident limitation: it only applies to abelian categories, while many of the categories of interest to $K$-theory are not abelian, but exact (for example, the $K$-theory of a ring is defined as $K(R)=K(R\mathsf{-proj})$, where $R\mathsf{-proj}$ denotes the exact category of finitely generated projective $R$-modules).

Various authors have followed this line of thought and successfully generalized Quillen's Localization Theorem for abelian categories to exact categories by requiring additional conditions on the Serre subcategory $\A$. Their results can be stated as follows.
\begin{theorem} \cite{Sch},
\cite{Cardenas}
Let $\B$ be an exact category, and $\A\subseteq\B$ a Serre subcategory; that is, a full subcategory closed under extension, subobjects and quotients. If in addition $\A$ is left or right s-filtering \cite{Sch}, 
or localizes $\B$ \cite{Cardenas}, then there exists an exact quotient category $\B/\A$ and an exact functor $\loc\colon\B\to\B/\A$ such that $$K(\A)\to K(\B)\xrightarrow{\loc} K(\B/\A)$$ is a homotopy fibration sequence. Furthermore, the functor $\B\xrightarrow{\loc}\B/\A$ is universal among exact functors $\B\to\C$ that vanish on $\A$, where $\C$ is exact.
\end{theorem}

These exact versions of the Localization Theorem certainly widen the range of applications, but may still be quite restrictive. For example, given a well-behaved ring $R$, one would like to be able to apply the Localization Theorem to the categories $R\mathsf{-proj}\subseteq R\mathsf{-mod}$ and obtain a long exact sequence relating $K(R)=K(R\mathsf{-proj})$ and $G(R)=K(R\mathsf{-mod})$. However, recalling that every finitely generated $R$-module is a quotient of a (finitely generated) free $R$-module, we see that $R\mathsf{-proj}$ is never a Serre subcategory unless $R$ is such that every finitely generated $R$-module is projective,\footnote{This will happen precisely when $R$ is a product of matrices over division rings.} in which case the long exact sequence is no longer necessary.

Our goal is then to prove an exact version of the Localization Theorem that allows for a different type of subcategory, other than Serre subcategories. Since this property is vital when proving that the quotient category $\B/\A$ is exact for each of the three versions of exact localization mentioned above, we take a different approach: instead of looking for an {\it exact} quotient category $\B/\A$ whose morphisms encode the vanishing of $\A$ on $K$-theory, we construct a {\it Waldhausen} category structure on $\B$ whose weak equivalences encode the vanishing of $\A$ on $K$-theory. Our main result is the following (see Theorem \ref{localization} and Theorem \ref{universal_prop}).

\begin{theorem}\label{main_thm_1}
Let $\B$ be an exact category closed under kernels of epimorphisms and with enough injective objects, and let $\A\subseteq\B$ be a full subcategory with the 2-out-of-3 property for short exact sequences and containing all injective objects. 
Then there exists a Waldhausen category structure on $\B$ with the admissible monomorphisms as the cofibrations, denoted $(\B, w_\A)$, such that 
$$K(\A)\to K(\B)\to K(\B,w_\A)$$ is a homotopy fibration sequence.
Furthermore, the functor $\B\xrightarrow{\id_\B} (\B,w_\A)$ is universal among exact functors $F\colon\B\to\C$ such that $0\to FA$ is a weak equivalence for each $A\in\A$, where $\C$ is a Waldhausen category satisfying the extension and saturation axioms.
\end{theorem}

As examples, we show our Localization Theorem can be applied to the inclusion $\Rproj\subset\Rmod$ when $R$ is a quasi-Frobenius ring (such as $\mathbb{Z}/n\mathbb{Z}, \Bbbk [G]$ for $\Bbbk$ a field and $G$ a finite group, or any finite dimensional Hopf algebra) in Example \ref{Frobenius}, and when $R$ is an Artin-Gorenstein ring in Example \ref{artin_gorenstein}.

We note that this result is not an extension of Quillen's Localization Theorem, but instead an alternate tool for producing homotopy fibration sequences, as any categories $\A$ and $\B$ satisfying the hypotheses of both Localization Theorems will necessarily satisfy $\A=\B$; we refer the reader to Remark \ref{alternate_way} for details.

It is well known that some computational tools are typically lost in the passage from exact categories to Waldhausen categories; for example, Quillen's Resolution Theorem, instrumental in the proof of the Fundamental Theorem of algebraic $K$-theory, does not translate to Waldhausen categories. However, the algebraic nature of our Waldhausen category $(\B, w_\A)$ allows us to recover a version of the Resolution Theorem for exact categories with weak equivalences obtained through our machinery; this is Theorem \ref{resolution}.

In order to construct this Waldhausen category structure on $\B$, we rely on the algebraic notion of cotorsion pairs. These are pairs $(\P,\I)$ of classes of objects in $\B$ which are orthogonal to each other with respect to the $\Ext^1_B$ functor, see Definition \ref{cotorsion_defn}.

The relation between cotorsion pairs and model categories was first introduced in \cite{BR07} and further explored in \cite{Hov02}. Essentially, Hovey shows that compatible pairs of cotorsion pairs on an abelian category $\A$ are in bijective correspondence with the abelian model category structures that can be defined on $\A$.

A study of the proof immediately reveals that each of the two cotorsion pairs involved is related to one of the factorization systems present in the model category; thus, both are truly necessary to determine a model category structure.

Inspired by Hovey's result, we take a step back from our goal of an exact Localization Theorem to study the relation between cotorsion pairs and Waldhausen categories in its full generality. In doing so, we are able to show that just one cotorsion pair (the one corresponding to the cofibrant objects in Hovey's result) is enough to determine a Waldhausen category structure on the full subcategory of cofibrant objects. Namely, we prove the following (see Theorem \ref{the_thm}; compare with \cite[Thm.\ 2.5]{Hov07}).

\begin{theorem}\label{main_thm_2}
Let $\E$ be an exact category, and $\Z$, $\C$ two full subcategories of $\E$ such that $\Z\cap\C$ is closed under extensions and cokernels of admissible monomorphisms in $\C$, and $\C$ is part of a complete hereditary cotorsion pair $(\C,\C^\bot)$. Assume also that  $\C^\bot\subseteq\Z$. 

Then $\C$ admits a Waldhausen category structure $(\C,w_\Z)$, with the $\C$-admissible monomorphisms as the cofibrations, and with the weak equivalences as the morphisms that factor as a $\C$-admissible monomorphism with cokernel in  $\C\cap\Z$ followed by a $\C$-admissible epimorphism with kernel in $\C^\bot$.
\end{theorem}

Our result recovers familiar Waldhausen structures, such as the structure on bounded chain complexes with quasi-isomorphisms for weak equivalences, but also creates new examples that cannot be produced by restricting to the category of cofibrant objects of a model category obtained from Hovey's machinery, since not every cotorsion pair can be promoted to a pair of compatible cotorsion pairs.

\subsection*{Notation}

We clarify some notational conventions that we use in this paper. Coproducts will be denoted by $+$, and pushouts by $A+_B C$; similarly for products, pushouts, and $\times$. Given morphisms $f\colon A\to B, g\colon C\to B$ and $h\colon D\to E$, we abuse notation and use $+$ to denote both the map $f+g\colon A+C\to B$ given by $(a,c)\mapsto f(a)+g(c)$, and the map $f+h:A+D\to B+E$ given by $(a,d)\mapsto (f(a),h(d))$.

Cofibrations will be denoted by $\cof$, admissible epimorphisms by $\twoheadrightarrow$, and weak equivalences by $\we$. Exact categories will be denoted by their underlying category without mention to the specified class of exact sequences. If $\C$ is an exact category, we use $\C$ and $(\C,\text{isos})$ interchangeably to refer to the usual Waldhausen structure in $\C$. We use $(\C,w_\Z)$ for the Waldhausen structure with weak equivalences determined through our result, where we omit mention to the (usual) class of cofibrations.

Given an exact category $\C$, we use $\C^\bot$ (resp.\ ${}^\bot \C$) for its left (resp.\ right) orthogonal complement with respect to the $\Ext^1$ functor (see Defn.\ \ref{cotorsion_defn}). Cotorsion pairs will be denoted by $(\P,\I)$ when dealt with in abstract, and by $(\C,\C^\bot)$ in our main results, to emphasize it being determined by $\C$.

\subsection*{Acknowledgements}

The author is deeply grateful to her advisor, Inna Zakharevich, for all of her encouragement, and for many useful discussions throughout the preparation of this paper; especially, for pointing out a key fact that led to the proof of the Resolution Theorem. The author would also like to thank Charles Weibel for thought-provoking conversations related to this work, and Daniel Grayson and an anonymous referee for their detailed reading and feedback.



\tableofcontents


\section{Preliminaries}

This section includes the definitions required to make our paper self-contained. None of the definitions are original, and the reader familiar with exact categories, Waldhausen categories, and cotorsion pairs, can safely omit this preliminary material.

\subsection{Exact and Waldhausen categories}

Exact categories are of the utmost importance to $K$-theory. In fact, this was the setting originally used by Quillen in \cite{Qui73} to develop the notions of higher algebraic $K$-theory and extend the known results of classical $K$-theory ($K_n$ for $n=0,1,2$). While Quillen gives an intrinsic definition of exact categories, we choose to work with an equivalent definition that makes use of an ambient abelian category. The curious reader can find a proof of the equivalence of both definitions in Remark 2.8, Lemma 10.20, and Theorem A.1(i) of \cite{Buh10}.

\begin{defn}
An {\it exact category} is a pair $(\C,E)$ where $\C$ is an additive category and $E$ is a family of sequences in $\C$ of the form \[ 0\to A\xrightarrow{i} B \xrightarrow{p} C\to 0\] such that there exists an embedding of $\C$ as a full subcategory of an abelian category $\A$, with the following properties:
\begin{enumerate}
\item[(E1)] $E$ is the class of all sequences in $\C$ which are exact in $\A$,
\item[(E2)] $\C$ is closed under extensions; that is, given an exact sequence in $\A$ as above, if $A,C\in\C$ then $B$ is isomorphic to an object in $\C$.
\end{enumerate}
\end{defn}
Morphisms appearing as the first map in a sequence in $E$ (like $i$ above) are called admissible monomorphisms, and those appearing as the second map (like $p$) are admissible epimorphisms. We typically drop the class of exact sequences from the notation, and refer to the exact category $(\C,E)$ as $\C$.

\begin{defn}
A subcategory $\Z$ of an exact category $\C$ is {\it closed under cokernels of admissible monomorphisms} if for every exact sequence in $\C$
\[ 0\to A\to B\to C\to 0\]
such that $A,B\in \Z$, we have $C\in\Z$. It is {\it closed under kernels of admissible epimorphisms} if, whenever $B,C\in\Z$, we have $A\in\Z$.
We say $\Z$ has {\it 2-out-of-3 for exact sequences} if whenever two of the objects $A,B,C$ in an exact sequence as above are in $\Z$, the third one is as well.
\end{defn}

A more general setting for algebraic $K$-theory was introduced by Waldhausen \cite{Wal85} under the name of ``categories with cofibrations and weak equivalences''; in the modern literature these are called Waldhausen categories.

\begin{defn}
A \emph{Waldhausen category} consists of a category $\C$ with a zero object, together with two subcategories $\co\C$ and $\weq\C$, whose maps are respectively called cofibrations and weak equivalences, satisfying the following axioms:
\begin{enumerate}
\item[(C1)] all isomorphisms in $\C$ are cofibrations,
\item[(C2)] for every object $A$ in $\C$, the map $0\to A$ is a cofibration,
\item[(C3)] if $A\hookrightarrow B$ is a cofibration, then for any map $A\to C$, the pushout
    \[\btk
    A\rar[hookrightarrow]\dar & B\dar\\
    C\rar & B+_A C
    \etk\] exists in $\C$ and the map $C\to B+_A C$ is a cofibration,
\item[(W1)] all isomorphisms in $\C$ are weak equivalences,
\item[(W2)] the ``Gluing Lemma'': given a commutative diagram
\[\btk
C\dar["\sim"] & A\lar\rar[hookrightarrow]\dar["\sim"] & B\dar["\sim"]\\
C' & A'\lar\rar[hookrightarrow] & B'
\etk\] where the horizontal arrows on the right are cofibrations and all vertical maps are weak equivalences, the induced map $$B+_A C\to B'+_{A'} C'$$ is also a weak equivalence.
\end{enumerate}
\end{defn}

The correct notion of functor between Waldhausen categories is that of an exact functor.

\begin{defn}
A functor between Waldhausen categories $F:\C\to\D$ is {\it exact} if it preserves the $0$ object, cofibrations, weak equivalences, and the pushout diagrams of axiom (C3).
\end{defn}

\begin{rmk}\label{exact_are_waldhausen}
Exact categories form a notable example of Waldhausen categories. Given an exact category $\C$, we can consider it as a Waldhausen category by letting cofibrations be the admissible monomorphisms, and weak equivalences be the isomorphisms. From this perspective, Quillen's $K$-theory space of the exact category $\C$ agrees with Waldhausen's $K$-theory space of the Waldhausen category $\C$.
\end{rmk}

\subsection{Cotorsion pairs}

Just like in abelian categories, one can use the structure present in an exact category $\C$ to construct the functor $\Ext^1_\C$, which is precisely the restriction to $\C$ of the functor $\Ext^1$ defined in any ambient abelian category $\A$. For any two objects $A,B\in \C$, $\Ext^1_\C(A,B)$ is the abelian group of equivalence classes of extensions of $A$ by $B$ in $\C$; that is, classes of sequences in $\C$ of the form \begin{equation}\label{seq}
0\to B\to C\to A\to 0.
\end{equation} As in the usual case, one can observe that $\Ext^1_\C(A,B)=0$ precisely when every sequence in $\C$ as in (\ref{seq}) is isomorphic to the canonical split exact sequence \[ 0\to B\to A + B\to A\to 0;\] that is, when every sequence as in (\ref{seq}) splits.

In the presence of an $\Ext$ functor, we can define cotorsion pairs. These were introduced in the late 70's by Salce under the name of ``cotorsion theories'' \cite{Salce}, and became more widely known in the 90's when Bican, El Bashir, and Enochs used them to prove the flat cover conjecture: namely, that for any ring $R$, all $R$-modules admit a flat cover \cite{Enochs_etal}. The definition is as follows.

\begin{defn}\label{cotorsion_defn}
A \emph{cotorsion pair} in an exact category $\C$ is a pair $(\P,\I)$ consisting of two classes of objects of $\C$ that are the orthogonal complement of each other with respect to the $\Ext^1_\C$ functor. More explicitly, if we let
$$\P^\bot\coloneqq\{A\in\C \text{ such that }\Ext^1_\C(P,A)=0 \text{ for every } P\in\P\}$$ and
$${}^\bot \I\coloneqq\{A\in\C \text{ such that }\Ext^1_\C(A,I)=0 \text{ for every } I\in\I\},$$
then it must be that $\P^\bot=\I$ and ${}^\bot \I=\P$.
\end{defn}

\begin{rmk} It is not hard to observe that both the left and right classes participating in a cotorsion pair must be closed under extensions. This implies that any class of objects participating in a cotorsion pair in $\C$ is an exact category, when considered as a full subcategory of $\C$. Henceforth, we make no distinction between $\P$ and $\Ob\P$ whenever $\P\subseteq\C$ is a full exact subcategory whose class of objects participates in a cotorsion pair.
\end{rmk}

Note that cotorsion pairs provide a generalization of injective and projective objects in an exact category; indeed, an object $P$ is projective in $\C$ precisely when the functor $\Hom_\C(P,-)$ is exact, which is equivalent to requiring that $\Ext^1_\C(P,A)=0$ for every $A\in\C$. Thus, if $\proj$ denotes the full subcategory of projective objects in $\C$, we see that $(\proj,\C)$ is a cotorsion pair. Similarly, if $\inj$ denotes the full subcategory of injective objects, then $(\C,\inj)$ is a cotorsion pair.

Borrowing motivation from the case of injectives and projectives, it is of interest to know when a cotorsion pair provides resolutions for any given object.

\begin{defn}
Let $(\P,\I)$ be a cotorsion pair in an exact category $\C$. We say that the cotorsion pair is \emph{complete} if any object $A$ in $\C$ can be resolved as
\[ 0\to A\to I\to P\to 0\] for some $I\in\I$, $P\in\P$, and as
\[ 0\to I'\to P'\to A\to 0\] for some $I'\in\I$, $P'\in\P$.
\end{defn}

As an example, we can see that the cotorsion pair $(\C,\inj)$ is complete precisely if $\C$ has enough injectives.

\begin{defn}
A complete cotorsion pair $(\P,\I)$ is {\it hereditary} if the category $\P$ is closed under kernels of admissible epimorphisms in $\C$, or equivalently, if $\I$ is closed under cokernels of admissible monomorphisms in $\C$.
\end{defn}

For further details on cotorsion pairs, we refer the reader to \cite{Enochs}.






\section{Results from Hovey}

In this section we define the classes of maps we will use throughout the paper, and we recall a collection of results from \cite{Hov02} that will be essential when proving our main theorem. Since Hovey deals with two compatible cotorsion pairs, each of the results below are a subset of the cited statements in his article; however, one can check that the proofs for these claims only make use of the amount of structure present in our hypotheses.\footnote{The reader might note another discrepancy between our hypotheses and Hovey's: he requires the cotorsion pairs to be {\it functorially} complete. A careful study of his constructions reveals that this extra condition is used to obtain {\it functorial} factorizations of maps, which we do not require.}

For the entirety of this section, let $\E$ be an exact category, and $\C$, $\Z$ be two full subcategories of $\E$ satisfying the conditions of Theorem \ref{main_thm_2}.
The cotorsion pair $(\C,\C^\bot)$ can be used to define the following distinguished classes of morphisms:

\begin{itemize}
\item a morphism in $\C$ is a {\it cofibration} ($\hookrightarrow$) if it is an admissible monomorphism with cokernel in $\C$,
\item a morphism in $\E$ is an {\it acyclic fibration}$\btk[column sep=small] (\rar[twoheadrightarrow,"\sim"] & )\etk$if it is an admissible epimorphism with kernel in $\C^\bot$.
\end{itemize}

With the addition of the category $\Z$, we can also define the following:

\begin{itemize}
\item a morphism in $\C$ is an {\it acyclic cofibration}$\btk[column sep=small] (\rar[hookrightarrow,"\sim"] & )\etk$if it is an admissible monomorphism with cokernel in $\C\cap\Z$,
\item a morphism in $\C$ is a {\it weak equivalence} ($\we$) if it factors as the composition of an acyclic cofibration followed by an acyclic fibration.
\end{itemize}



The first three results we recall involve only the complete cotorsion pair $(\C,\C^\bot)$, and the classes of maps determined from it.

\begin{lemma}\label{compositions}\cite[Lemma 5.3]{Hov02}
Let $i\colon A\cof B$ and $j\colon B\cof C$ be two cofibrations. 
Then, the composition $ji\colon A\to C$ is also a cofibration. Similarly, the composition of two acyclic fibrations is again an acyclic fibration.
\end{lemma}



As we mentioned before, in Hovey's proof each cotorsion pair is related to one of the weak factorization systems present in a model category. In our setting, we retain one of the pairs, and thus we still have one factorization system. In particular, this implies the following two results.

\begin{prop}\label{liftings}\cite[Prop.\ 4.2]{Hov02}
A map is an admissible monomorphism with cokernel in $\C$ if and only if it has the left lifting property with respect to all acyclic fibrations.
\end{prop}

\begin{prop}\label{factorization}\cite[Prop.\ 5.4]{Hov02}
 Every map $f\colon A\to B$ between objects of $\C$ can be factored as $f=pi$, where $i$ is a cofibration and $p$ an acyclic fibration.
 \end{prop}
\begin{rmk}\label{epis_in_C}
Let $f\colon A\to B$ be a map between objects of $\C$, and $A\xrightarrow{i} C \xrightarrow{p} B$ a factorization of $f$ as above. 
Then the short exact sequence
\[\btk 0\rar & \ker p\rar & C\rar["p"] & B \rar & 0,\etk\] together with the fact that $\C$ is hereditary, imply that $\ker p\in\C$. Therefore, when factoring morphisms in $\C$, we can always assume that $p$ is an acyclic fibration which is furthermore an admissible epimorphism in $\C$.
\end{rmk}

Finally, the last three results involve the structure determined by the cotorsion pair, together with the category $\Z$. Note that we only require $\Z\cap\C$ to be closed under extensions and cokernels of admissible monomorphisms, while Hovey asks these properties of $\Z$ (denoted by $\mathcal{W}$ in his paper), and, in addition, that $\Z$ also be closed under retracts and kernels of admissible epimorphisms. Because of these missing properties, some of Hovey's results will no longer hold in our general setting; notably, we won't have 2-out-of-3 for the class of weak equivalences. However, the following still hold.

\begin{lemma}\label{we_cofib}\cite[Lemma 5.8]{Hov02}
A map is an acyclic cofibration if and only if it is a cofibration and a weak equivalence. Similarly, an $\E$-admissible epimorphism between objects of $\C$ with kernel in $\Z$ is a weak equivalence.
\end{lemma}

\begin{prop}\label{we}\cite[Prop.\ 5.6]{Hov02}
Weak equivalences are closed under composition.
\end{prop}

\begin{prop}\label{2-out-of-3}\cite[Lemmas 5.9, 5.10]{Hov02}
Let $f$ and $g$ be two composable maps in $\C$ such that $g$ and $gf$ are weak equivalences. Then $f$ is also a weak equivalence, in the following cases:
\begin{itemize}
\item $g$ and $gf$ are acyclic fibrations, or
\item $g$ is an acyclic fibration and $gf$ is an acyclic cofibration.
\end{itemize}
\end{prop}

\section{Waldhausen categories from cotorsion pairs}

We now proceed to prove one of our main results, which we state once again, and whose proof consists of the entirety of this section. It explains how to produce a Waldhausen category structure from a cotorsion pair and a chosen subcategory $\Z$ which will form the class of acyclic objects.

\begin{thm}\label{the_thm}
Let $\E$ be an exact category, and $\Z$, $\C$ two full subcategories of $\E$ such that $\Z\cap\C$ is closed under extensions and cokernels of admissible monomorphisms in $\C$, and $\C$ is part of a complete hereditary cotorsion pair $(\C,\C^\bot)$. Assume also that  $\C^\bot\subseteq\Z$. 

Then $\C$ admits a Waldhausen category structure $(\C,w_\Z)$, with the $\C$-admissible monomorphisms as the cofibrations, and with the weak equivalences as the morphisms that factor as a $\C$-admissible monomorphism with cokernel in  $\C\cap\Z$ followed by a $\C$-admissible epimorphism with kernel in $\C^\bot$.
\end{thm}

Before checking the axioms, we must pay attention to a few things. First, note that $\C$ contains the zero object of $\E$, since $\Ext^1_\E(0,A)=0$ for any object $A\in\E$.
Also note that our cofibrations and weak equivalences form subcategories, which is ensured by Lemma \ref{compositions} and Proposition \ref{we}.



We now verify that the axioms of a Waldhausen category are satisfied. As we will see, most of them follow directly from the definitions, and the difficulty arises from the Gluing Lemma.

\begin{lemma}[C1]
Isomorphisms are cofibrations.
\end{lemma}
\begin{proof}
If $f$ is an isomorphism, then it is an admissible monomorphism with $\cok f=0\in\C$.
\end{proof}

\begin{lemma}[C2]
$0\to A$ is a cofibration for every $A\in\C$.
\end{lemma}
\begin{proof}
 Let $A\in\C$; the morphism $0\to A$ is an admissible monomorphism, and its cokernel is $A\in\C$.
\end{proof}

\begin{lemma}[C3]
Cofibrations are closed under cobase change.
\end{lemma}
\begin{proof}
Let $i\colon A\cof B$ be a cofibration and $f\colon A\to C$ any map in $\C$. We know the pushout
\[\btk
 A\rar[hookrightarrow,"i"]\dar["f"'] & B\dar\\
    C\rar & B+_A C
    \etk\]
 exists in $\E$. Moreover, the map $C\to B+_A C$ will be an admissible monomorphism, since these are preserved by pushouts, and furthermore, $$\cok(C\to B+_A C)=\cok i\in\C .$$ Lastly, the short exact sequence
 \[\btk[column sep=small] 0\rar & C\rar & B+_A C\rar & \cok i \rar & 0\etk\] implies that $B+_A C$ is an object of $\C$, since $\C$ is closed under extensions.
\end{proof}

In fact, one can similarly see that the dual result also holds.

\begin{lemma}\label{pullback}
Let $p\colon A\twoheadrightarrow B$ be an acyclic fibration between objects in $\C^\bot$, and $f\colon C\to B$ be any map with $C\in\C^\bot$. Then the pullback $A\times_B C$ is in $\C^\bot$, and $A\times_B C\to C$ is an acyclic fibration.
\end{lemma}

\begin{lemma}[W1]
Isomorphisms are weak equivalences.
\end{lemma}
\begin{proof}
  An isomorphism $f\colon A\to B$ factors trivially as $f=1_B f$.
\end{proof}

We now show two special instances of the Gluing Lemma (W2).

\begin{prop}\label{gluing_for_cofibs}
Consider a commutative diagram in $\C$
\[\btk
C\dar[hookrightarrow,"\gamma", "\sim"'] & A\lar["j"']\rar[hookrightarrow,"i"]\dar[hookrightarrow,"\alpha","\sim"'] & B\dar[hookrightarrow,"\beta","\sim"']\\
C' & A'\lar["j'"]\rar[hookrightarrow,"i'"'] & B'
\etk\] such that the induced map $\hat{i}\colon\cok\alpha\to\cok\beta$ is a cofibration. Then the map on pushouts $\phi\colon B+_A C\to B'+_{A'}C'$ is an acyclic cofibration.
\end{prop}
\begin{proof}
Consider the following commutative diagram with exact rows in $\C$ 
\\
\[
\btk[bo column sep, bo row sep]
 & C\ar[rrd,hookrightarrow,"\iota_C"]\ar[ddd,"\gamma" near start,"\sim"' near start,hookrightarrow] & & & & \\
 A\ar[rrd,hookrightarrow,"i" pos=0.6, crossing over]\ar[ddd,"\alpha" near start,"\sim"' near start,hookrightarrow] \ar[ur,"j"] & &            & B+_A C\ar[ddd,"\phi" near start] \ar[rrd,twoheadrightarrow] & &  \\
  & & B\ar[rrd, crossing over, twoheadrightarrow]
  \ar[ur,"\iota_B"] & & & \cok\iota_C\ar[ddd,"\hat{\phi}" near start]  \\
 & C'\ar[rrd,"\iota_{C'}"', hookrightarrow] & & & \cok i
 \ar[ur,equal] & \\
 A'\ar[rrd, hookrightarrow, "i'"]\ar[ur,"j'"] & & & B'+_{A'} C'\ar[rrd, twoheadrightarrow] & &  \\
 & & B'\ar[rrd, twoheadrightarrow]\ar[ur,"\iota_{B'}"]\ar[uuu,"\beta"' near end,"\sim" near end,hookleftarrow, crossing over] & & &\cok\iota_{C'} \\
 & & & & \cok i'\ar[ur,equal]\ar[uuu, leftarrow, "\hat{\beta}"' near end, crossing over]&
\etk
\]

Applying the Snake Lemma to the front face, we get the following exact sequence in an ambient abelian category $\A$
\begin{equation}\label{eq1}\btk[column sep=small] 0\rar & \ker\hat{\beta}\rar & \cok\alpha\rar["\hat{i}"] & \cok\beta\rar & \cok\hat{\beta}\rar & 0.\etk\end{equation} By assumption, the map $\hat{i}$ is a cofibration; in particular, it is a monomorphism, and therefore $\ker\hat{\beta}=0$.

Now, applying the Snake Lemma to the back face yields the exact sequence in $\A$
\begin{equation}\label{eq2}\btk[column sep=small] 0\rar & \ker\phi\rar & \ker\hat{\phi}\rar & \cok\gamma\rar & \cok\phi\rar & \cok\hat{\phi}\rar & 0.\etk\end{equation} Since $\hat{\phi}=\hat{\beta}$, we get that $\ker\hat{\phi}=0$ and thus $\ker\phi=0$, proving $\phi$ is a monomorphism. It remains to show that $\cok\phi\in\C\cap\Z$.

First, note from (\ref{eq1}) that $\cok\hat{i}= \cok\hat{\beta}$; then, since $\hat{i}$ is a cofibration, we have $\cok\hat{\beta}\in\C$.
 Also, sequence (\ref{eq1}) actually reduces to
 \[\btk[column sep=small] 0\rar & \cok\alpha\rar & \cok\beta\rar & \cok\hat{\beta}\rar & 0,\etk\] and because $\alpha$ and $\beta$ are acyclic cofibrations, we know that $\cok\alpha,\cok\beta\in\Z\cap\C$; using the fact that $\Z\cap\C$ is closed under cokernels of admissible monomorphisms in $\C$, we conclude that $\cok\hat{\beta}\in\Z$.

 Finally, we can reduce sequence (\ref{eq2}) to
  \[\btk[column sep=small] 0\rar & \cok\gamma\rar & \cok\phi\rar & \cok\hat{\phi}\rar & 0\etk\]
 from which we see that $\cok\phi$ is an extension of $\cok\hat{\phi} (=\cok\hat{\beta})$ by $\cok\gamma$, both of which belong to $\C\cap\Z$; therefore $\cok\phi\in\C\cap\Z$, concluding our proof.
\end{proof}

We can also show the corresponding result when the vertical maps are acyclic fibrations. Although the idea of the proof is similar in spirit, we include it in order to demonstrate the need for the cotorsion pair $(\C,\C^\bot)$ to be hereditary.

\begin{prop}\label{gluing_for_fibs}
Given a commutative diagram in $\C$
\[\btk
C\dar[twoheadrightarrow,"\gamma", "\sim"'] & A\lar["j"']\rar[hookrightarrow,"i"]\dar[twoheadrightarrow,"\alpha","\sim"'] & B\dar[twoheadrightarrow,"\beta","\sim"']\\
C' & A'\lar["j'"]\rar[hookrightarrow,"i'"'] & B'
\etk\] 
the map on pushouts $\phi\colon B+_A C\to B'+_{A'}C'$ is a weak equivalence.
\end{prop}
\begin{proof}
Consider the commutative diagram from the proof of Proposition \ref{gluing_for_cofibs}, where the vertical maps are now acyclic fibrations. Applying the Snake Lemma to the front face yields the exact sequence
\[\btk[column sep=small] 0\rar & \ker\alpha\rar & \ker\beta\rar & \ker\hat{\beta}\rar & 0\etk\]
Note that this is an exact sequence in $\C$, since $\C$ is closed under kernels of $\E$-admissible epimorphisms. Also, $\ker\alpha,\ker\beta\in\C^\bot\subseteq\Z$ because $\alpha$ and $\beta$ are acyclic fibrations, and thus $\ker\hat{\beta}\in\Z$ due to the fact that $\Z\cap\C$ is closed under cokernels of admissible monomorphisms in $\C$.

Now, the Snake Lemma for the back face yields the exact sequence in $\A$
\[\btk[column sep=small] 0\rar & \ker\gamma\rar & \ker\phi\rar & \ker\hat{\phi}\rar & 0\rar & \cok\phi \rar & \cok\hat{\phi}\rar & 0.\etk\] Looking at its right part, we see that $\cok\phi\cong\cok\hat{\phi}$; however, $\cok\hat{\phi}=\cok\hat{\beta}$ and $\hat{\beta}$ is an admissible epimorphism (since $\beta$ is); hence $\phi$ is also an  epimorphism.

Finally, looking at the left part of the sequence we see that $\ker\phi$ is an extension of $\ker\hat{\phi}(=\ker\hat{\beta})$ by $\ker\gamma$, both of which belong to $\Z\cap\C$, so $\ker\phi\in\Z$, proving that $\phi$ is an admissible epimorphism in $\E$, and moreover, a weak equivalence due to Proposition \ref{we_cofib}.
\end{proof}

These two propositions, along with Theorem \ref{spans} in the appendix, allow us to prove the Gluing Lemma.

\begin{thm}[W2, Gluing Lemma]\label{gluing}
Given a commutative diagram in $\C$
\[\btk
C\dar["\sim"',"\gamma"] & A\lar["j"']\rar[hookrightarrow,"i"]\dar["\sim"',"\alpha"] & B\dar["\sim"',"\beta"]\\
C' & A'\lar["j'"]\rar[hookrightarrow,"i'"'] & B'
\etk\] the induced map $B+_A C\to B'+_{A'} C'$ is a weak equivalence.
\end{thm}
\begin{proof}
The maps $\alpha, \beta$ and $\gamma$ are weak equivalences, so by definition they can be expressed as the composition of an acyclic cofibration followed by an acyclic fibration, as shown in the diagram below left.
\begin{equation}\label{eq6}\btk
C\dar["\sim"',"\gamma_1", hookrightarrow] & A\lar["j"']\rar[hookrightarrow,"i"]\dar["\sim"',"\alpha_1", hookrightarrow] & B\dar["\sim"',"\beta_1", hookrightarrow]\\
\overline{C}\dar["\sim"',"\gamma_2", twoheadrightarrow] & \overline{A}\dar["\sim"',"\alpha_2", twoheadrightarrow] & \overline{B}\dar["\sim"',"\beta_2", twoheadrightarrow]\\
C' & A'\lar["j'"]\rar[hookrightarrow,"i'"'] & B'
\etk \hspace{1.5cm}
\btk
C\dar["\sim"',"\gamma_1", hookrightarrow] & A\lar["j"']\rar[hookrightarrow,"i"]\dar["\sim"',"\alpha_1", hookrightarrow] & B\dar["\sim"',"\beta_1", hookrightarrow]\\
\overline{C}\dar["\sim"',"\gamma_2", twoheadrightarrow] & \overline{A}\dar["\sim"',"\alpha_2", twoheadrightarrow]\rar["\overline{i}"]\lar["\overline{j}"'] & \overline{B}\dar["\sim"',"\beta_2", twoheadrightarrow]\\
C' & A'\lar["j'"]\rar[hookrightarrow,"i'"'] & B'
\etk\end{equation}
Using the fact that cofibrations have the left lifting property with respect to acyclic fibrations (Proposition \ref{liftings}), the liftings in the two squares below allow us to complete the middle horizontal row of our diagram above right.
\[\btk
A\rar["\beta_1 i"]\dar["\alpha_1"', hookrightarrow] & \overline{B}\dar["\beta_2", "\sim"', twoheadrightarrow]\\
\overline{A}\ar[ur,"\overline{i}", dashed]\rar["i'\alpha_2"'] & B'
\etk \hspace{2cm}
\btk
A\rar["\gamma_1 j"]\dar["\alpha_1"', hookrightarrow] & \overline{C}\dar["\gamma_2", "\sim"', twoheadrightarrow]\\
\overline{A}\rar["j'\alpha_2"']\ar[ur, dashed,"\overline{j}"] & C'
\etk\]

A priori, there is no reason for the map $\overline{i}$ to be a cofibration. However, if we factor $\overline{i}=pl$ where $l$ is a cofibration and $p$ an acyclic fibration, we can once again use the lifting property in the squares
\[\btk
A\rar["l\alpha_1"]\dar["i"', hookrightarrow] & \hat{B}\dar["p", "\sim"', twoheadrightarrow]\\
B\ar[ur,"\delta_1", dashed]\rar["\beta_1"'] & \overline{B}
\etk \hspace{2cm}
\btk
\overline{A}\rar["i'\alpha_2"]\dar["l"', hookrightarrow] & B'\dar[equal]\\
\hat{B}\rar["\beta_2 p"']\ar[ur, dashed,"\delta_2"] & B'
\etk\]
to obtain the diagram
\[\btk
A\dar["\alpha_1", hookrightarrow,"\sim"']\rar["i",hookrightarrow] & B\rar[equal]\dar["\delta_1"]& B\dar["\beta_1","\sim"', hookrightarrow]\\
\overline{A}\rar["l", hookrightarrow]\dar["\alpha_2","\sim"',twoheadrightarrow] & \hat{B}\rar["p","\sim"', twoheadrightarrow]\dar["\delta_2"]
& \overline{B}\dar["\beta_2","\sim"', twoheadrightarrow]\\
A'\rar["i'",hookrightarrow] & B'\rar[equal] & B'
\etk\]

Since $p\delta_1=\beta_1$, where $p$ is an acyclic fibration and $\beta_1$ an acyclic cofibration, we see that $\delta_1$ must be a monomorphism. Furthermore, we can apply the Snake Lemma to the diagram
\[\btk
0\rar & 0\rar\dar & B\rar[equal]\dar["\delta_1"] & B\dar[hookrightarrow,"\beta_1"]\rar & 0\\
0\rar & \ker p\rar & \hat{B}\rar["p"] & \overline{B}\rar & 0
\etk\]
and obtain the short exact sequence in $\E$
\[\btk
0\rar & \ker p\rar & \cok\delta_1\rar & \cok\beta_1\rar & 0.
\etk\] Here $\ker p\in\C^\bot\cap\C$ (since $p$ is a $\C$-admissible epimorphism; see Remark \ref{epis_in_C}), and $\cok\beta_1\in\C$; thus $\cok\delta_1\in\C$ proving $\delta_1$ is a cofibration. Moreover, it is an acyclic cofibration, by Proposition \ref{2-out-of-3}.

Also, since $\delta_2=\beta_2 p$, with $p$ and $\beta_2$ acyclic fibrations, $\delta_2$ is an acyclic fibration by Proposition \ref{compositions}. Furthermore, $l$ is a cofibration, and we have $\delta_2\delta_1=\beta_2\beta_1$. This means we can safely assume the map $\overline{i}$ obtained in (\ref{eq6}) is a cofibration.

We now wish to find a way to obtain the stronger hypotheses of Proposition \ref{gluing_for_cofibs}; the key idea is to work with $\gamma_i, \alpha_i$, and $\beta_i$ as a single map in the category $\Span(\E)$. We refer the reader to Appendix \ref{appendix} for a definition of the category of spans, and its relation to our treatment of cotorsion pairs.

Given that $(\C,\C^\bot)$ is a complete cotorsion pair in $\E$, Theorem \ref{spans} ensures that $(\C_\Sp,\C^\bot_\Sp)$ is a complete cotorsion pair in $\Span(\E)$. Then, by Proposition \ref{factorization}, the maps of spans $(\gamma_1,\alpha_1,\beta_1)$ and $(\gamma_2,\alpha_2,\beta_2)$ in the rightmost diagram of figure (\ref{eq6}) factor as a cofibration followed by an acyclic fibration in $\Span(\E)$; that is, they factor as
\begin{equation}\label{eq7}\btk
C\dar[hookrightarrow,,"\gamma_1^1"] & A\lar["j"']\rar[hookrightarrow,"i"]\dar[hookrightarrow,,"\alpha_1^1"] & B\dar[hookrightarrow,,"\beta_1^1"]\\
C''\dar[twoheadrightarrow,"\sim"',"\gamma_1^2"] & A''\lar["j''"']\rar[hookrightarrow,"i''"] \dar[twoheadrightarrow,"\sim"',"\alpha_1^2"] & B''\dar[twoheadrightarrow,"\sim"',"\beta_1^2"]\\
\overline{C} & \overline{A}\lar["\overline{j}"']\rar[hookrightarrow,"\overline{i}"] & \overline{B}
\etk \hspace{1.5cm}
\btk
\overline{C}\dar[hookrightarrow,,"\gamma_2^1"] & \overline{A}\lar["\overline{j}"']\rar[hookrightarrow,"\overline{i}"] \dar[hookrightarrow,,"\alpha_2^1"] & \overline{B}\dar[hookrightarrow,,"\beta_2^1"]\\
C'''\dar[twoheadrightarrow,"\sim"',"\gamma_2^2"] & A'''\lar["j'''"']\rar[hookrightarrow,"i'''"] \dar[twoheadrightarrow,"\sim"',"\alpha_2^2"] & B'''\dar[twoheadrightarrow,"\sim"',"\beta_2^2"]\\
C' & A'\lar["j'"']\rar[hookrightarrow,"i'"] & B'
\etk\end{equation} where the induced maps $\cok\alpha_i^1\to\cok\beta_i^1$ are cofibrations, and $\ker\alpha_i^2\to\ker\gamma_i^2$ are acyclic fibrations, for $i=1,2$.\footnote{Note that the middle rows automatically have  cofibrations as their right maps, as every time we factor a map between objects of $\P_\Sp$, the middle object in the factorization also belongs to $\P_\Sp$.}

Moreover, we have that $\alpha_1^2\alpha_1^1=\alpha_1$, where $\alpha_1^2$ is an acyclic fibration and $\alpha_1$ an acyclic cofibration; thus, by Proposition \ref{2-out-of-3}, $\alpha_1^1$ is an acyclic cofibration (and similarly for $\gamma_1^1$ and $\beta_1^1$). On the other hand, $\alpha_2^2\alpha_2^1=\alpha_2$, where $\alpha_2^2$ and $\alpha_2$ are acyclic fibrations; thus $\alpha_2^1$ is an acyclic cofibration (and similarly for $\gamma_2^1$ and $\beta_2^1$).


Now we can apply Proposition \ref{gluing_for_cofibs} to the top half of both diagrams in figure (\ref{eq7}) to get that the maps $B+_A C\to B''+_{A''}C''$ and $\overline{B}+_{\overline{A}} \overline{C}\to B'''+_{A'''}C'''$ are acyclic cofibrations. We can also apply Proposition \ref{gluing_for_fibs} to the bottom half of both diagrams to get that $B''+_{A''}C''\to\overline{B}+_{\overline{A}}\overline{C}$ and $B'''+_{A'''}C'''\to B'+_{A'}C'$ are weak equivalences. Since the composition of these four pushout maps is clearly the map $B+_A C\to B'+_{A'}C'$, we see that this last map is a composition of weak equivalences, and thus a weak equivalence itself.
\end{proof}

It should be pointed out that many of the difficulties when proving the Gluing Lemma arise from the fact that our class of weak equivalences need not be saturated; that is, they need not satisfy 2-out-of-3.  Indeed, if they were saturated, then one could find a proof of the Gluing Lemma, for example, in  \cite[Thm.\ 2.27]{KP97}. However, we wished to investigate the broadest possible relation between cotorsion pairs and $K$-theory, and saturation will not hold under our general assumptions on the subcategory $\Z$. As we will see later on in Proposition \ref{saturation}, the lack of this property would be solved by requiring that $\Z\cap\C$ have 2-out-of-3 for exact sequences in $\C$.


\section{Properties satisfied by $(\C,w_\Z)$}

Waldhausen categories obtained from cotorsion pairs enjoy many desired properties, which we now study. Throughout this section we continue to use the notation and hypotheses of Theorem \ref{the_thm}.

As suggested by the notation, the category $\Z$ used to define the weak equivalences consists of the acyclic objects in our Waldhausen category.

\begin{lemma}\label{acyclic_obj}
For any object $A$ of $\C$, the map $0\to A$ 
is a weak equivalence if and only if $A\in\Z$.
\end{lemma}
\begin{proof}
If $A\in\Z$, then the map $0\to A$ is an admissible monomorphism with cokernel $A\in\C\cap\Z$; thus it is an acyclic cofibration, and in particular, a weak equivalence.

If the map $0:0\to A$ is a weak equivalence, we factor it as \[0=0\xrightarrow{i} \overline{A}\xrightarrow{p} A\] where $i$ is an acyclic cofibration and $p$ an acyclic fibration. Then $\overline{A}=\cok i\in\Z\cap\C$, and the exact sequence
\[0\to \ker p\to \overline{A}\xrightarrow{p} A \to 0\] has its two leftmost terms in $\Z\cap\C$; hence $A\in\Z$, since $\Z\cap\C$ is closed under cokernels of admissible monomorphisms in $\C$.
\end{proof}

\begin{lemma}[Left and right properness]\label{proper} Weak equivalences are stable under pushout with a cofibration, and under pullback with an admissible epimorphism.
\end{lemma}
\begin{proof}
Let $i:A\hookrightarrow B$ be a cofibration and $a:A\xrightarrow{\sim} A'$ a weak equivalence, which factors as $a=p_a i_a$ for some acyclic fibration $p_a$ and acyclic cofibration $i_a$. Taking successive pushouts, we can consider the diagram
\[\btk
0\rar & A\rar["i", hookrightarrow]\dar[hookrightarrow,"i_a","\sim"'] & B\rar\dar[hookrightarrow,"b_1"] & \cok i\rar\dar[equal] & 0\\
0\rar & \overline{A}\rar[hookrightarrow]\dar["p_a","\sim"',twoheadrightarrow] & \overline{A}+_A B\rar \dar[twoheadrightarrow, "b_2"]& \cok i\rar\dar[equal] & 0\\
0\rar & A'\rar["i", hookrightarrow] & A'+_{\overline{A}} (\overline{A}+_A B)\rar & \cok i\rar & 0
\etk\]

Applying the Snake Lemma to the top diagram, we see that $\cok i_a\cong \cok b_1$, and thus $b_1$ is also an acyclic cofibration. Similarly, the Snake Lemma on the bottom diagram yields $\ker p_a\cong \ker b_2$, and so $b_2$ is an acyclic fibration. Therefore, the composition
\[\btk
B\rar["b_2b_1"] & A'+_{\overline{A}} (\overline{A}+_A B)\simeq A'+_A B
\etk\] must be a weak equivalence.

The statement for admissible epimorphisms proceeds by duality.
\end{proof}

\begin{defn}
Following \cite{Wal85}, we say a Waldhausen category satisfies the {\it extension axiom} if any map between exact sequences
\[\btk
0\rar & A\rar[hookrightarrow]\dar["a","\sim"'] & B\rar\dar["b"] & C\rar\dar["c","\sim"'] & 0\\
0\rar & A'\rar[hookrightarrow] & B'\rar & C'\rar & 0
\etk\] where $a$ and $c$ are weak equivalences is such that $b$ is also a weak equivalence.
\end{defn}

\begin{prop}[Extension]\label{extension} Any Waldhausen category obtained through a cotorsion pair satisfies the extension axiom.
\end{prop}
\begin{proof}
Consider a map of exact sequences $(a,b,c)$ as above. By \cite[Prop.\ 3.1]{Buh10}, this map can be factored as
\[\btk
0\rar & A\rar[hookrightarrow]\dar["a","\sim"'] & B\rar\dar["b_1"] & C\rar\dar[equal] & 0\\
0\rar & A'\rar[hookrightarrow]\dar[equal] & P\rar\dar["b_2"] & C\rar\dar["c","\sim"'] & 0\\
0\rar & A'\rar[hookrightarrow] & B'\rar & C'\rar & 0
\etk\] where the top left and bottom right squares are bicartesian. Then, Lemma \ref{proper} ensures $b_1$ and $b_2$ are weak equivalences, and thus $b=b_2b_1$ is, too.
\end{proof}




\begin{defn} A Waldhausen category satisfies the {\it saturation} axiom if, given composable maps $f, g$, whenever two of $f, g$ and $gf$ are weak equivalences, so is the third.
\end{defn}

\begin{prop}[Saturation]\label{saturation} A Waldhausen category obtained through a cotorsion pair satisfies the saturation axiom if and only if $\Z\cap\C$ has 2-out-of-3 for exact sequences in $\C$.
\end{prop}
\begin{rmk}
Note that we always assume the subcategory $\Z\cap\C$ is closed under extensions and cokernels of admissible monomorphisms in $\C$. Therefore, $\Z\cap\C$ has 2-out-of-3 for short exact sequences when, in addition to the hypotheses required in out theorem, $\Z\cap\C$ is closed under kernels of admissible epimorphisms in $\C$.
\end{rmk}
\begin{proof}
Assume $\Z\cap\C$ has 2-out-of-3 for exact sequences in $\C$, and let $f,g$ be composable maps in $\C$. Then, whenever $gf$ and $g$ are weak equivalences, $f$ is one as well, since the proof in \cite[Lemma 5.11]{Hov02} applies verbatim.

If instead $gf$ and $f$ are weak equivalences, we cannot apply Hovey's proof to conclude that $f$ is a weak equivalence, since it makes use of a factorization system no longer present in our setting; thus, we appeal to a different argument.

We first show a special case: let $j$ and $k$ be composable cofibrations, and suppose $j$ and $kj$ are acyclic cofibrations; we will show $k$ is one as well. Consider the following diagram
\[\btk
0\rar &  A\dar[hookrightarrow, "j"]\rar[equal] & A\dar[hookrightarrow,"kj"]\rar & 0\rar\dar & 0\\
0\rar & B\rar[hookrightarrow,"k"] & C\rar & \cok k\rar & 0
\etk\] This yields the exact sequence in $\C$
$$0\to \cok j\to \cok kj\to \cok k\to 0$$ whose two leftmost terms are in $\Z$; thus, $\cok k\in\Z$ and so $k$ is a weak equivalence.

For the general case, let $gf$ and $f$ be weak equivalences, and factor $g=pi$, with $p$ an acyclic fibration and $i$ a cofibration. Then $gf=pif$ with $gf$ and $p$ weak equivalences, so by the above instance of saturation, we get that $if$ is a weak equivalence. Since $f$ is a weak equivalence, it admits a factorization $f=qj$ as an acyclic cofibration followed by an acyclic fibration. Now factor $iq=rk$ as a cofibration followed by an acyclic fibration. We have that $if=iqj=rkj$ is a weak equivalence, and so is $r$, thus using the above instance of saturation once again, we see that $kj$ is a weak equivalence. But now $j$ and $kj$ are acyclic cofibrations, with $k$ a cofibration as well, and so $k$ must be an acyclic cofibration. Therefore $iq=rk$ is a weak equivalence. The diagram
\[\btk
0\rar & \ker q\dar["\hat{k}"]\rar[hookrightarrow] & A\dar[hookrightarrow,"k","\sim"']\rar[twoheadrightarrow,"\sim","q"'] & B\rar\dar[hookrightarrow,"i"] & 0\\
0\rar & \ker r\rar[hookrightarrow] & D\rar[twoheadrightarrow,"\sim","r"'] & C\rar & 0
\etk\] yields the exact sequence in an ambient abelian category $\A$
$$0\to \cok\hat{k}\to \cok k\to \cok i\to 0$$

Since $\C$ is closed under kernels of admissible epimorphisms, this is actually an exact sequence in $\C$. To show $i$ is a weak equivalence, it suffices to prove that $\cok\hat{k}\in\Z$, because we already have $\cok k\in\Z$. Indeed, in the exact sequence in $\C$
$$0\to\ker q\xrightarrow{\hat{k}} \ker r\to \cok\hat{k}\to 0,$$ the two leftmost terms belong to $\Z$, since $q$ and $r$ are acyclic fibrations. This shows that $\cok\hat{k}\in\Z$ and thus that $i$ is an acyclic cofibration, which in turn implies $g$ is a weak equivalence.

For the converse, suppose the Waldhausen category we obtain is saturated, and let $$0\to A\xrightarrow{i} B\to C\to 0$$ be an exact sequence in $\C$ with $B,C\in\Z$. Then $i$ is an acyclic cofibration, and $0\to B$ is a weak equivalence by Lemma \ref{acyclic_obj}. We thus have
\[\btk[column sep=small, row sep=small]
0\ar[rr,"\sim"]\ar[dr] & & B\\
& A\ar[ur,"i"',"\sim"] &
\etk\] and saturation implies $0\to A$ is a weak equivalence as well; then $A\in\Z$ by Lemma \ref{acyclic_obj}.
\end{proof}

\section{The Localization Theorem}

This section is devoted to the proof of our second main result: an exact version of Quillen's Localization Theorem, as advertised in the introduction. To make things clearer, we separate the proof of the fibration sequence from that of the universal property.

\begin{thm}\label{localization}
Let $\B$ be an exact category closed under kernels of epimorphisms and with enough injective objects, and $\A\subseteq\B$ a full subcategory having 2-out-of-3 for short exact sequences and containing all injective objects.
Then there exists a Waldhausen category $(\B, w_\A)$ with admissible monomorphisms as cofibrations, such that
$$K(\A)\to K(\B)\to K(\B,w_\A)$$ is a homotopy fibration sequence.
\end{thm}
\begin{proof}
Apply Theorem \ref{the_thm} to the exact category $\E=\B$, by considering the subcategories $\C=\B$ and $\Z=\A$. Note that $\B$ is always part of a cotorsion pair $(\B,\B^\bot)$ with respect to the functor $\Ext^1_\B$; in this case, $\B^\bot=\inj$, the subcategory of injective objects in $\B$. Since $\B$ has enough injective objects, the cotorsion pair $(\B,\inj)$ is complete; furthermore, it is hereditary because $\B$ is assumed to be closed under kernels of epimorphisms. By assumption, $\inj\subseteq\A$, and $\A$ has 2-out-of-3 for exact sequences, so, in particular, it is closed under extensions and cokernels of admissible monomorphisms.

Let $(\B,w_\A)$ denote the Waldhausen category obtained through Theorem \ref{the_thm}; by construction this has admissible monomorphisms as cofibrations. Also, Lemma \ref{acyclic_obj} shows that $B^{w_\A}=\A$, where $\B^{w_\A}$ is the standard notation for the full subcategory of $\B$ consisting of those objects $A\in\B$ such that $0\to A$ is a weak equivalence in $(\B,w_\A)$.

Since $(\B,w_\A)$ is such that every map factors as a cofibration followed by a weak equivalence by Proposition \ref{factorization} and Lemma \ref{we_cofib}, and it satisfies the extension (Proposition \ref{extension}) and saturation (Proposition \ref{saturation}) axioms, we can apply Schlichting's cylinder-free version of Waldhausen's fibration theorem \cite[Thm.\ A.3]{Sch06} to the inclusion $(\B,\text{isos})\subset (\B,w_\A)$ to get the desired homotopy fibration sequence $$K(\A)\to K(\B)\to K(\B,w_\A)$$
\end{proof}

\begin{rmk}\label{alternate_way}
Note that our Localization Theorem above is not an extension of Quillen's Localization Theorem, but rather an alternate way of obtaining homotopy fibration sequences from inclusions $\A\subseteq\B$.

Indeed, suppose $\A$ and $\B$ are in the hypotheses of both Localization Theorems, and let $X\in\B$ be any object. Since $\B$ has enough injectives, there exists an embedding $X\hookrightarrow I$ into an injective object. Then $I\in\A$, since it contains all injective objects. However, $\A$ is a Serre subcategory, and therefore closed under subobjects; thus $X\in\A$ and we see that $\A=\B$.
\end{rmk}

\begin{thm}\label{universal_prop}
The functor $\B\xrightarrow{\id_\B} (\B,w_\A)$ is universal among exact functors $F\colon\B\to\C$ such that $0\to FA$ is a weak equivalence for each $A\in\A$, where $\C$ is a Waldhausen category satisfying the extension and saturation axioms.
\end{thm}
\begin{proof}
Let $\C$ be a Waldhausen category satisfying the extension and saturation axioms, and $F\colon\B\to\C$ an exact functor. In order to prove the result, it suffices to show that the functor $F\colon (\B,w_\A)\to\C$ is exact. Since $\B$ and $(\B,w_\A)$ have the same underlying category and the same cofibrations, we need only show that $F$ takes weak equivalences in $(\B,w_\A)$ to weak equivalences in $\C$. Recalling that weak equivalences in $(\B,w_\A)$ factor as an acyclic cofibration followed by an acyclic fibration, we show that $F$ takes these two classes of morphisms to weak equivalences in $\C$.

Let $i\colon A\xhookrightarrow{\sim} B$ be an acyclic cofibration in $(\B,w_\A)$. Since $F$ is exact, it preserves exact sequences, and we can consider the following diagram in $\C$
\[\btk
0\rar & FA\rar[equal]\dar[equal] & FA\rar\dar[hookrightarrow,"Fi"] & 0\rar\dar & 0\\
0\rar & FA\rar[hookrightarrow,"Fi"] & FB\rar & \cok Fi\rar & 0
\etk\]  However, $\cok i\in\A$ since $i$ is an acyclic cofibration. By assumption, this implies that $0\to F\cok i$ is a weak equivalence; but $F\cok i\cong\cok Fi$ and thus $Fi$ is a weak equivalence by the extension axiom on $\C$.

Now, let$\btk[column sep=small] p:A\rar[twoheadrightarrow,"\sim"] & B\etk$be an acyclic fibration, and consider the diagram
\[\btk
0\rar & F\ker p\rar\dar & FA\dar["Fp"]\rar["Fp"] & FB\rar\dar[equal] & 0\\
0\rar & 0\rar & FB\rar[equal] & FB\rar & 0
\etk\] We have $\ker p\in\A$, so by assumption $0\to F\ker p$ is a weak equivalence. Thus, the saturation axiom applied to the maps
\[\btk[column sep=small, row sep=small]
& 0\ar[dl,"\sim"']\ar[dr,equal] & \\
F\ker p\ar[rr] & & 0
\etk\] tells us the map $F\ker p\to 0$ is a weak equivalence; the extension axiom then implies that $Fp$ is a weak equivalence.
\end{proof}

\section{The Resolution Theorem}

In this section, we show that Quillen's Resolution Theorem, valid for exact categories with isomorphisms as weak equivalences, also has a natural formulation in our setting of more general weak equivalences.

Let $\C$ be an exact category. Recall that a {\it resolution} of an object $A$ is a sequence of maps in $\C$ $$\dots\to B_n\xrightarrow{d_n} B_{n-1}\to\dots\to B_1\xrightarrow{d_1} B_0\xrightarrow{d_0} A$$ such that for each $n\geq0$, the map $d_n$ factors as
\[\btk[column sep=small, row sep=small]
B_n\ar[rr,"d_n"]\ar[dr,"p_n"', twoheadrightarrow] & & B_{n-1}\\
& Z_n\ar[ur,"i_n"', hookrightarrow] &
\etk\] where $B_{-1}\coloneqq A, p_0=d_0, i_0=1_A$, and $$0\to Z_{n+1}\xrightarrow{i_{n+1}} B_n\xrightarrow{p_n} Z_n\to 0$$ is an exact sequence in $\C$. Given a subcategory $\P\subseteq\C$, we say the above is a {\it $\P$-resolution} if $B_n\in\P$ for each $n\geq 0$.

We begin by stating Quillen's original theorem.

\begin{thm}
Let $\P$ be a full exact subcategory of an exact category $\C$, such that $\P$ is closed under extensions and kernels of admissible epimorphisms in $\C$. If every object $A\in\C$ admits a finite $\P$-resolution $$0\to P_n\to\cdots\to P_1\to P_0\to A\to 0,$$ then $K(\P)\simeq K(\C)$.
\end{thm}

This result can be adapted to our setting as follows.

\begin{thm}\label{resolution}
Let $\E$ be an exact category, and $\C, \Z$ two full subcategories of $\E$ such that $\C$ is closed under kernels of admissible epimorphisms and is part of a complete cotorsion pair $(\C,\C^\bot)$, with $\C^\bot\subseteq\Z$, and such that $\Z\cap\C$ has 2-out-of-3 for short exact sequences in $\C$.

Let $\P$ be a full subcategory of $\C$ such that $\P$ is closed under extensions and kernels of admissible epimorphisms in $\C$. In addition, assume that the cotorsion pair $(\P,\P^\bot)$ (defined with respect to $\Ext^1_P$) 
is complete, and that $\P^\bot\subseteq\Z$. If every object $A\in\C$ admits a finite $\P$-resolution $$0\to P_n\to\cdots\to P_1\to P_0\to A\to 0,$$ then $K(\P, w_{\Z\cap\P})\simeq K(\C,w_\Z)$.
\end{thm}

Before proceeding with the proof of this result, note that the full subcategory $\Z\cap\P$ has 2-out-of-3 for exact sequences in $\P$, since these are also exact in $\C$, where $\Z$ has the 2-out-of-3 property. 
This fact, together with the additional assumptions that $\P$ is closed under kernels of admissible epimorphisms, that the pair $(\P,\P^\bot)$ is complete and that $\P^\bot\subseteq\Z$ (and thus, since $\P^\bot\subseteq\P$ by definition, we have $\P^\bot\subseteq\Z\cap\P$) allow us to apply Theorem \ref{the_thm} to define the aforementioned Waldhausen category structure on the exact category $\P$.

Finally, one can verify that the inclusion $i\colon\P\hookrightarrow\C$ is an exact functor, although curiously it does not exhibit $\P$ as a Waldhausen subcategory of $\C$.

\begin{proof}
Let $(\C,w_\Z)$ denote the Waldhausen category structure on $\C$ with weak equivalences given by Theorem \ref{the_thm}, and $(\C,\text{isos})$ denote the usual Waldhausen category structure on $\C$ with weak equivalences given by isomorphisms.

By Proposition \ref{acyclic_obj}, we know that $\Z\cap\C$ is precisely the full subcategory of objects $A$ in $\C$ such that the map $0\to A$ is a weak equivalence; thus, $\Z\cap\C=\C^{w_\Z}$, where the latter is the commonly used notation for this class.

Since both $(\C,w_\Z)$ and $(\P,w_{\Z\cap\P})$ are such that every map factors as a cofibration followed by a weak equivalence by Proposition \ref{factorization} and Lemma \ref{we_cofib}, and they satisfy the extension (Proposition \ref{extension}) and saturation (Proposition \ref{saturation}) axioms, we can apply Schlichting's cylinder-free version of Waldhausen's fibration theorem \cite[Thm.\ A.3]{Sch06} to get the following diagram, whose horizontal rows are homotopy fibrations
\[\btk
K(\Z\cap\P,\text{isos})\rar\dar["k"] & K(\P,\text{isos})\rar\dar["j"] & K(\P,w_{\Z\cap\P})\dar["i"]\\
K(\Z\cap\C,\text{isos})\rar & K(\C,\text{isos})\rar & K(\C,w_\Z)
\etk\]

The functor $j$ is a homotopy equivalence, as given by Quillen's Resolution Theorem. Then, if we consider the $K$-theory spectra, it suffices to show that $k$ induces a homotopy equivalence as well, since in the stable case the homotopy fibration sequences are also homotopy cofibration sequences, and thus the two cofibers would be uniquely determined up to homotopy. To achieve this, we show that $(\Z\cap\P,\text{isos})$ and $(\Z\cap\C,\text{isos})$ also satisfy the hypotheses of the Resolution Theorem; that is, we must check that every object in $\Z\cap\C$ admits a finite resolution by objects in $\Z\cap\P$.

Let $A$ be an object in $\Z\cap\C$; then, as an object in $\C$, it admits a finite $\P$-resolution
$$0\to P_n\to\cdots\to P_1\to P_0\to A\to 0.$$ That means there exist exact sequences in $\C$
$$0\to P_n\hookrightarrow P_{n-1}\twoheadrightarrow C_{n-1}\to 0$$
$$0\to C_{n-1}\hookrightarrow P_{n-2}\twoheadrightarrow C_{n-2}\to 0$$
$$\vdots$$
$$0\to C_2\hookrightarrow P_1\twoheadrightarrow C_1\to 0$$
$$0\to C_1\hookrightarrow P_0\twoheadrightarrow A\to 0$$

Since $(\P,\P^\bot)$ is complete, there exists an exact sequence $$0\to P_n\hookrightarrow Z_n\twoheadrightarrow C'\to 0$$ with $Z_n\in\P^\bot$ and $C'\in\P$. Recall that $\P^\bot\subset\P$, since this cotorsion pair is defined with respect to the functor $\Ext^1_\P$, and that also, by assumption, we have $\P^\bot\subseteq\Z$; thus, $Z_n\in\Z\cap\P$. Consider the diagram
\[\btk
0\rar & P_n\rar[hookrightarrow]\dar[hookrightarrow] & P_{n-1}\dar[hookrightarrow]\rar[twoheadrightarrow] & C_{n-1}\dar[equal]\rar & 0\\
0\rar & Z_n\rar[hookrightarrow] & Z_n+_{P_n} P_{n-1}\rar[twoheadrightarrow] & C_{n-1}\rar & 0
\etk\] Rename $Q\coloneqq Z_n+_{P_n} P_{n-1}$; we can similarly find a resolution $$0\to Q\hookrightarrow Z_{n-1}\twoheadrightarrow C''\to 0$$ with $Z_{n-1}\in\P^\bot\subseteq\Z\cap\P$ and $C''\in\P$. Then, we construct the diagram
\[\btk
0\rar & Z_n\dar[equal]\rar[hookrightarrow] & Q\dar[hookrightarrow]\rar[twoheadrightarrow] & C_{n-1}\dar[hookrightarrow]\rar & 0\\
0\rar & Z_n\rar[hookrightarrow] & Z_{n-1}\rar[twoheadrightarrow] & Z_{n-1} +_{Q}C_{n-1} \rar & 0
\etk\] by first taking the pushout square on the right; since $Q\hookrightarrow Z_{n-1}$ is a monomorphism, this is also a pullback square, and thus we get the pictured identity map between the kernels of the two horizontal admissible epimorphisms.

Rename $D_{n-1}\coloneqq Z_{n-1} +_{Q}C_{n-1}$; given that $\Z\cap\C$ has 2-out-of-3 for exact sequences, we have $D_{n-1}\in\Z$. Now consider the diagram
\[\btk
0\rar & C_{n-1}\rar[hookrightarrow]\dar[hookrightarrow] & P_{n-2}\dar[hookrightarrow]\rar[twoheadrightarrow] & C_{n-2}\dar[equal]\rar & 0\\
0\rar & D_{n-1}\rar[hookrightarrow] & D_{n-1}+_{C_{n-1}} P_{n-2}\rar[twoheadrightarrow] & C_{n-2}\rar & 0
\etk\] Rename $Q_{n-2}\coloneqq D_{n-1}+_{C_{n-1}} P_{n-2}$. Since pushouts preserve cokernels, we see that
\begin{align*}
\cok(P_{n-2}\hookrightarrow Q_{n-2}) & \simeq \cok(C_{n-1}\hookrightarrow D_{n-1})\\
& \simeq \cok(Q\hookrightarrow Z_{n-1})\\
& \simeq C''\in\P
\end{align*} and thus, using the fact that $\P$ is closed under extensions and that $P_{n-2},C''\in\P$, we get that $Q_{n-2}\in\P$.

We can now use the completeness of the cotorsion pair $(\P,\P^\bot)$ again, to get an admissible monomorphism $Q_{n-2}\hookrightarrow Z_{n-2}$ for some $Z_{n-2}\in\Z\cap\P$. Then, we construct the diagram
\[\btk
0\rar & D_{n-1}\dar[equal]\rar[hookrightarrow] & Q_{n-2}\dar[hookrightarrow]\rar[twoheadrightarrow] & C_{n-2}\dar[hookrightarrow]\rar & 0\\
0\rar & D_{n-1}\rar[hookrightarrow] & Z_{n-2}\rar[twoheadrightarrow] & Z_{n-2} +_{Q_{n-2}}C_{n-2} \rar & 0
\etk\] where the square on the right is cocartesian. Naming $D_{n-2}\coloneqq Z_{n-2} +_{Q_{n-2}}C_{n-2}$ and using the fact that $\Z\cap\C$ has 2-out-of-3 for exact sequences, we see that $D_{n-2}\in\Z$.

Repeating this process, we obtain exact sequences in $\Z$
$$0\to Z_n\hookrightarrow Z_{n-1}\twoheadrightarrow D_{n-1}\to 0$$
$$0\to D_{n-1}\hookrightarrow Z_{n-2}\twoheadrightarrow D_{n-2}\to 0$$
$$\vdots$$
$$0\to D_2\hookrightarrow Z_1\twoheadrightarrow D_1\to 0$$
where $Z_i\in\Z\cap\P$ for every $i$.

For the final step, consider the diagram
\[\btk
0\rar & C_1\rar[hookrightarrow]\dar[hookrightarrow] & P_0\dar[hookrightarrow]\rar[twoheadrightarrow] & A\dar[equal]\rar & 0\\
0\rar & D_1\rar[hookrightarrow] & D_1+_{C_1} P_0\rar[twoheadrightarrow] & A\rar & 0
\etk\] If we let $Z_0\coloneqq D_1+_{C_1} P_0$, we can use the same reasonings as above to see that $Z_0\in\Z\cap\P$, and hence obtain our finite $\Z\cap\P$-resolution in $\Z$
$$0\to Z_n\to\dots\to Z_1\to Z_0\to A\to 0.$$
\end{proof}

\section{Examples}

\begin{ex}[Chain complexes]\label{chain_complexes}
We can use this presentation to recover the usual Waldhausen structure on bounded chain complexes over an exact category with enough injectives.

Let $\D$ be an exact category, which is a full subcategory of some abelian category $\A$.
Then $\Ch(\D)$ is a Waldhausen category by letting cofibrations be the chain maps that are degreewise admissible monomorphisms in $\D$, and weak equivalences be the quasi-isomorphisms (as seen in the ambient abelian category $\Ch(\A)$\footnote{With some work, this definition can be made independent of the ambient abelian category.}); see \cite[II, 9.2]{Wei13}.

To obtain this from a cotorsion pair, we let $\E=\C=\Ch(\D)$. The subcategory $\Z$ should be that of acyclic objects; in this case, we wish these to be precisely the exact chain complexes. Now, if we denote by $\inj$ the class of injective objects in $\Ch(\D)$, we have that $(\Ch(\D),\inj)$ is a cotorsion pair, which will be complete as long as $\D$ has enough injectives. Note that all injective objects in $\Ch(\D)$ must be exact, and that $\Z$ is  closed under extensions and cokernels of admissible monomorphisms.

Using Theorem \ref{the_thm} we get a Waldhausen category structure on $\Ch(\D)$, where cofibrations are admissible monomorphisms and weak equivalences are maps that factor as an acyclic cofibration followed by an acyclic fibration. It remains to show that these coincide with the quasi-isomorphisms. To see that this is so, recall the following well-known result, which is Exercise 1.3.5 in \cite{Wei94}.

\begin{lemma}\label{exact_ker_coker}
Let $f\colon A\to B$ be a map in $\Ch(\D)$. If $\ker f$ and $\cok f$ are both exact as elements in $\Ch(\A)$, then $f$ is a quasi-isomorphism.
\end{lemma}



Now, suppose $f$ is a map that factors as $f=pi$, for some acyclic cofibration $i$ and acyclic fibration $p$. Then $i$ and $p$ are maps having exact kernel and cokernel, so by Lemma \ref{exact_ker_coker} they are both quasi-isomorphisms; hence, so is $f$.
Conversely, let $f\colon A\to B$ be a quasi-isomorphism. Due to Proposition \ref{factorization}, we can factor $f$ as $A\xhookrightarrow{i} C\xrightarrow{p} B$ where $i$ is a cofibration and $p$ an acyclic fibration. Using Lemma \ref{exact_ker_coker} once more, we see that $p$ is a quasi-isomorphism; then, since quasi-isomorphisms have 2-out-of-3, $i$ must also be a quasi-isomorphism. Since $i$ is a monomorphism, its cokernel must be exact and thus $i$ is an acyclic cofibration.
\end{ex}

\begin{ex}[Chain complexes with degreewise-split cofibrations]\label{ex_split}
For any exact category $\D$, the category of chain complexes $\Ch(\D)$ admits a different Waldhausen category structure, with quasi-isomorphisms as weak equivalences but with cofibrations the degreewise split admissible monomorphisms.

Even if $\D$ does not have enough injectives, we can obtain this structure from a cotorsion pair as well, by considering $\E=\C=\Ch(\D)$ as an exact category whose class of exact sequences is given by degreewise split exact sequences; denote this by $\Chdw(\D)$. In this case, one can show that $(\Chdw(\D),\contr)$ is a cotorsion pair, where $\contr$ denotes the subcategory of contractible complexes (i.e., of complexes $X$ such that $1_X$ is null-homotopic).

Furthermore, this cotorsion pair is complete, since for every complex $X$, its cone $C(X)$ is contractible and $X\hookrightarrow C(X)$ is degreewise split. Letting $\Z$ be the class of exact complexes, we see that $\contr\subseteq\Z$ if $\D$ is assumed to be idempotent complete \cite[Prop.\ 10.9]{Buh10}, and the same reasoning as in the previous example shows that the weak equivalences we obtain through Theorem \ref{the_thm} are precisely the quasi-isomorphisms.
\end{ex}

\begin{rmk}
A fundamental fact is that any exact category $\D$ can be considered a Waldhausen category, where the cofibrations are the admissible monomorphisms and the weak equivalences are the isomorphisms.

This example cannot be built from a cotorsion pair. To do so, we would need to set $\C=\D$, and $\Z$ to be the full subcategory of objects isomorphic to $0$. However, it is seldom the case that all $\D$-injective objects are isomorphic to zero; thus, in general $\D^\bot\not\subseteq\Z$ regardless of the ambient exact category $\E$.

This is somewhat disappointing but not entirely insurmountable: even though it is not possible to define the Waldhausen category $\D$ through a cotorsion pair, the Gillet-Waldhausen theorem \cite[V, Thm.\ 2.2]{Wei13} shows that, whenever $\D$ is closed under kernels of epimorphisms, $K(\D)\simeq K(\Ch(\D))$. Hence, if $\D$ is closed under kernels of epimorphisms and has enough injectives, the $K$-theory spectrum of $\D$ is always equivalent to that of $\Ch(\D)$, which can be defined from a cotorsion pair, as shown in Example \ref{chain_complexes}.

If $\D$ does not have enough injectives or is not closed under kernels of epimorphisms, this can still be done, by recalling that $$K(\D)\simeq K(\hat{\D})\simeq K(\Ch(\hat{\D}))\simeq K(\Chdw(\hat{\D})),$$ where $K(\Chdw(\hat{\D}))$ is obtained through a cotorsion pair as in Example \ref{ex_split}. Here $\hat{\D}$ denotes the full exact subcategory of the idempotent completion of $\D$ consisting of the objects $A$ such that $[A]\in K_0(\D)$. The first equivalence is given by \cite[IV, Ex.\ 8.13]{Wei13}, the second one is due to the Gillet-Waldhausen theorem since $\hat{\D}$ is always closed under kernels of epimorphisms (see \cite[IV, Ex.\ 8.13]{Wei13}), and the third equivalence is given by \cite[V, Ex.\ 2.3]{Wei13}.
\end{rmk}

\begin{ex}
It is possible to use Hovey's result \cite[Thm.\ 2.5]{Hov07} to obtain a model category, and later restrict to a Waldhausen category on its small cofibrant objects; however, not every Waldhausen category determined from a cotorsion pair comes from an abelian model category in this manner.

First of all, note that we do not need the subcategory $\Z$ to have 2-out-of-3 in order to construct a Waldhausen category, while this will always be the case for the class of acyclic objects in a model category; thus, this gives a simple way to recognize examples that do not come from restricting the structure present in a model category. 
However, it is also possible to construct examples where $\Z$ has 2-out-of-3 and the Waldhausen category cannot be promoted to a model category via Hovey's result.

To see this, let $\A$ be an abelian category with enough injectives, and consider the Waldhausen category obtained from the cotorsion pair $(\A,\inj)$, where acyclic objects are those of finite injective dimension. 
For this to come from an abelian model category, there should exist a category of fibrant objects $\F$ such that $\inj=\F\cap\Z$ and that $(\Z,\F)$ is a complete cotorsion pair. But the left class in any cotorsion pair must contain the class of projective objects, and therefore whenever $\A$ has a projective object of infinite injective dimension, $(\Z,\F)$ cannot be a cotorsion pair regardless of the category $\F$.

For a concrete instance of this, let $\Bbbk$ be a field and consider the $\Bbbk$-algebra $A_n=\Bbbk Q_n/I_n$, where $Q_n$ is the quiver
\[\btk
0\ar[out=210, in=150,loop,"\alpha_0"]\rar["\alpha_1"] & 1\rar["\alpha_2"] & 2\rar["\alpha_2"] & \cdots\rar["\alpha_{n-1}"] & n-1\rar["\alpha_n"] & n
\etk\] and $I_n=\langle \alpha_0^2, \alpha_1\alpha_0, \alpha_2\alpha_1, \dots, \alpha_n\alpha_{n-1}\rangle$. For any $n\geq 1$, the projective (and simple) left $A$-module
\[\btk
P_n: 0\rar & 0\rar & 0\rar & \cdots\rar & 0\rar & \Bbbk
\etk\] has infinite injective dimension.
\end{ex}

\begin{ex}\label{Frobenius}
Let $R$ be a quasi-Frobenius ring, that is, a ring such that the classes of projective and injective $R$-modules agree, and denote by $\Rmod$ the category of finitely generated right $R$-modules. We can consider the cotorsion pair $(\Rmod, \Rinj)=(\Rmod, \Rproj)$, where $\Rinj=\Rproj$ is the full subcategory of finitely generated injective-projective $R$-modules.

Every $R$-module in a quasi-Frobenius ring can be embedded in a free $R$-module; 
thus, in particular, every finitely generated $R$-module can be embedded in a finitely generated projective $R$-module. This shows the cotorsion pair $(\Rmod, \Rproj)$ is complete. It is hereditary since quasi-Frobenius rings are Noetherian.

Also, the class $\Z=\Rproj=\Rinj$ has 2-out-of-3 for exact sequences, since both classes are always closed under extensions, the category of projective modules over any ring is always closed under kernels of admissible epimorphisms, and the category of injective modules is closed under cokernels of admissible monomorphisms.

This implies we can apply Theorem \ref{the_thm} to get a Waldhausen structure on $\Rmod$ with admissible monomorphisms as cofibrations, and weak equivalences the maps that factor as an (admissible) monomorphism with projective cokernel followed by an (admissible) epimorphism with projective kernel.
Furthermore, Theorem \ref{localization} yields a homotopy fibration sequence
$$K(R)\to K(\Rmod)\to K(\Rmod, w_{\Rproj})$$ where the rightmost term considers the Waldhausen category structure described above, and the other two terms compute Quillen's $K$-theory of the exact categories with isomorphisms. Thus, in a way, $K(\Rmod, w_{\Rproj})$ measures the difference between $K(R)$ and $G(R)$.

Examples of quasi-Frobenius rings are $\mathbb{Z}/n\mathbb{Z}, \Bbbk [G]$ for $\Bbbk$ a field and $G$ a finite group, and any finite dimensional Hopf algebra.
\end{ex}




\begin{ex}\label{artin_gorenstein}
Let $R$ be an Artin algebra that is also a Gorenstein ring (this is, $R$ has finite injective dimension as a left and right module over itself). The full subcategory $\CM\subseteq\Rmod$ of maximal Cohen-Macaulay modules consists of those $M\in\Rmod$ such that there exists an exact sequence $$0\to M\to P^0\xrightarrow{d^0} P^1\xrightarrow{d^1} P^2\to\dots$$ with each $P^n\in\Rproj$ and $\ker d^n\in {}^\bot\Rproj$, i.e., $\Ext^1_{\Rmod}(\ker d^n,P)=0$ for every $P\in\Rproj$. When $R$ is commutative, this is the usual class of maximal Cohen-Macaulay modules, given by the finitely generated $R$-modules $M$ such that $\depth (M)=\dim (M)$.

Denote by $\Rproj^{<\infty}$ the class of finitely generated $R$-modules of finite projective dimension. Then $$(\CM, \Rproj^{<\infty})$$ is a complete hereditary cotorsion pair with respect to $\Ext^1_{\Rmod}$, and if we restrict to $\Ext^1_{\CM}$ we get the complete hereditary cotorsion pair $$(\CM,\Rproj^{<\infty} \cap\CM)=(\CM, \Rproj)$$ in $\CM$ \cite[VI, \S 3]{BR07}.

Note that in this case, the subcategory $\Rproj$ has 2-out-of-3 for exact sequences; thus, letting $\B=\CM$ and $\A=\Rproj$, Theorem \ref{localization} yields a homotopy fibration sequence $$K(R)\to K(\CM)\to K(\CM, w_{\Rproj})$$

Moreover, it is also known that for this class of rings, every finitely generated $R$-module admits a finite resolution by maximal Cohen-Macaulay modules \cite[VI, \S 2]{BR07}; thus Quillen's Resolution Theorem gives $$K(\CM)\simeq K(\Rmod).$$ We conclude that, for this class of rings, $K(\CM,w_{\Rproj})$ measures the difference between $K(R)$ and $G(R)$.

\end{ex}


\section{Appendix: Cotorsion pairs in the category of spans}\label{appendix}

The aim of this technical appendix is to show that complete cotorsion pairs in an exact category $\C$ induce complete cotorsion pairs in the category $\Span(\C)$ of spans in $\C$. More specifically, we will prove the following result.


\begin{thm}\label{spans}
Let $(\P,\I)$ be a complete cotorsion pair in an exact category $\C$. If we define two classes of objects in $\Span(\C)$ by \[\P_\Sp =\{C\leftarrow A\xhookrightarrow{i} B : A,B,C\in\P \text{ and } i \text{ is a cofibration}\}\] and 
\[\btk[column sep=small] \I_\Sp =\{C & A\lar[twoheadrightarrow,"\sim","p"']\rar & B : A,B,C\in\I \text{ and } p \text{ is an acyclic fibration}\}\etk,\] then $(\P_\Sp,\I_\Sp)$ is a complete cotorsion pair in $\Span(\C)$.
\end{thm}

Recall that, for any category $\D$, the category of spans over $\D$ is defined as the category of functors \[\Span(\D)\coloneqq [\bullet\leftarrow \bullet\to\bullet,\D].\] Thus $\Span(\D)$ has as objects all diagrams in $\D$ of shape $\bullet\leftarrow \bullet\to\bullet$, and natural transformations between them as morphisms.

If $\A$ is an abelian category, then $\Span(\A)$ is also abelian (as is any category of functors from a small category into $\A$), and thus if $\C\subseteq\A$ is an exact category embedded in $\A$, we see that $\Span(\C)$ is an exact category embedded in $\Span(\A)$.

We prove theorem \ref{spans} in two stages. For ease of notation, we denote the bifunctor $\Ext^1_{\Span(\C)}$ simply by $\Ext^1$.

\begin{thm}
If $(\P,\I)$ is a cotorsion pair in $\C$, then $(\P_\Sp,\I_\Sp)$ as defined above is a cotorsion pair in $\Span(\C)$.
\end{thm}
\begin{proof}
Fix spans \[\btk[column sep=small] P: P_C & P_A\lar["g"']\rar[hookrightarrow,"f"] & P_B  \in\P_\Sp\etk\] and \[\btk[column sep=small] I: I_C &  I_A\lar["g'"', twoheadrightarrow,"\sim"]\rar["f'"] & I_B  \in\I_\Sp;\etk\] we  show that $\Ext^1(P,I)=0$. Recall that elements in $\Ext^1(P,I)$ correspond to isomorphism classes of extensions of $P$ by $I$ in $\Span(\C)$; thus, proving $\Ext^1(P,I)=0$ is equivalent to showing that every extension of $P$ by $I$ in $\Span(\C)$ is split, and therefore isomorphic to the trivial extension
$0\to I\to P+I\to P\to 0$.

Consider an extension of $P$ by $I$ in ${\Span(\C)}$ as displayed below
\begin{equation}\label{seq_in_span}
\btk
0\rar & I_B\rar["\beta_2"] & B\rar["\beta_1"] & P_B\rar & 0\\
0\rar & I_A\dar["g'",twoheadrightarrow,"\sim"']\rar["\alpha_2"]\uar["f'"'] & A\dar["g''"]\uar["f''"']\rar["\alpha_1"] & P_A\dar["g"]\uar["f"', hookrightarrow]\rar & 0\\
0\rar & I_C\rar["\gamma_2"] & C\rar["\gamma_1"] & P_C\rar & 0
\etk
\end{equation} Since $\Ext^1_\C(P_C,I_C)=0$, we know the bottom sequence splits, and thus there exists a map $s\colon C\to I_C$ such that $s\gamma_2=1_{I_C}$.\footnote{In fact, each of the horizontal sequences splits in $\C$, for this same reason. However, there is no guarantee that the given splittings will assemble into a map in $\Span(\C)$; that is, we don't know the resulting diagram will commute.} Consider the commutative square
\[\btk
I_A\dar["\alpha_2"']\rar[equal] & I_A\dar["g'",twoheadrightarrow,"\sim"']\\
A\rar["sg''"'] & I_C
\etk\] We know $g'$ is an acyclic fibration, and $\alpha_2$ is an admissible monomorphism with cokernel in $\P$, so by Proposition \ref{liftings} there exists a lift in the square above, which we denote by $t\colon A\to I_A$.

We immediately see that $t$ defines a splitting of the short exact sequence in the middle. Moreover, if we consider the bottom half of our diagram (\ref{seq_in_span}) as a short exact sequence in the category of arrows $\Ar(\C)$, the condition $g't=sg''$ implies that $(t,s)$ defines a splitting of this sequence in $\Ar(\C)$. Therefore, there exist maps $t'\colon P_A\to A$ and $s'\colon P_C\to C$ such that $(t',s')$ also defines a splitting of that sequence in $\Ar(\C)$.

Finally, considering the diagram
\[\btk
P_A\rar["f''t'"]\dar["f"', hookrightarrow] & B\dar["\beta_1", twoheadrightarrow, "\sim"']\\
P_B\rar[equal] & P_B
\etk\] and its lift $t''\colon P_B\to B$, we see that $t''$ yields a splitting of the top exact sequence in such a way that $(t'',t',t)$ is a map in $\Span(\C)$, as desired.

Now, let $X: C\xleftarrow{g} A\xrightarrow{f} B$ be an element in $\Span(C)$ such that $\Ext^1(X,I)=0$ for every $I\in\I_\Sp$. We must show that $X\in\P_\Sp$; that is, that $A,B,C\in\P$ and $f$ is a cofibration.

To show that $A\in\P$, it suffices to prove that $\Ext^1_\C(A, J)=0$ for any $J\in\I$. Let
\[\btk[column sep=small] 0\rar & J\rar["\alpha_2"] & D\rar["\alpha_1"] & A\rar & 0\etk\] be an extension of $A$ by $J$ in $\C$. We can fit it into the following commutative diagram with exact rows
\[\btk
0\rar & 0\rar & B\rar[equal] & B\rar & 0\\
0\rar & J\dar\rar["\alpha_2"]\uar & D\dar["g\alpha_1"]\uar["f\alpha_1"']\rar["\alpha_1"] & A\dar["g"]\uar["f"']\rar & 0\\
0\rar & 0\rar & C\rar[equal] & C\rar & 0
\etk\] where the column on the left is an element of $\I_\Sp$; then, by assumption, this exact sequence in $\Span(\C)$ splits, and in particular, the middle sequence splits in $\C$. Hence every extension of $A$ by $J$ is split, for any $J\in\I$; this shows that $A\in\P$. Similarly, one shows that $B$ and $C$ belong to $\P$.

It remains to prove that $f$ is a cofibration.
By Proposition \ref{liftings}, it is enough to show that $f$ has the right lifting property with respect to all acyclic fibrations. As a first step towards this, we restrict ourselves to a smaller class of commutative squares and show that any diagram
\[\btk
A\dar["f"']\rar["\alpha"] & A'\dar[twoheadrightarrow,"\sim"', "p"]\\
B\rar[equal] & B
\etk\] admits a lift.

To see this, note that we can fit the data of the above square into the following short exact sequence in $\Span(\C)$
\[\btk
0\rar & \ker p\rar & A'\rar[twoheadrightarrow,"p"] & B\rar & 0\\
0\rar & 0\uar\dar\rar & A\rar[equal]\uar["\alpha"']\dar["g"] & A\uar["f"']\dar["g"]\rar & 0\\
0\rar & 0\rar & C\rar[equal] & C\rar & 0
\etk\] But $p$ is an acyclic fibration, so $\ker p\in\I$ and hence the left column is an element of $\I_\Sp$; 
thus, by assumption, this sequence splits. In particular, there exists a map $s\colon B\to A'$ such that $ps=1_B$ and $\alpha=sf$, providing the desired lift.

If instead we start with a general square
\[\btk
A\dar["f"']\rar["\alpha"] & A'\dar[twoheadrightarrow,"\sim"', "p"]\\
B\rar["\beta"'] & B'
\etk\] we construct the pullback displayed below left,
 \[\btk
B\times_{B'} A'\dar["\pi_B"']\rar["\pi_{A'}"] & A'\dar[twoheadrightarrow,"\sim"', "p"]\\
B\rar["\beta"'] & B'
\etk \hspace{2cm}
\btk
A\dar["f"']\rar["\varphi"] & B\times_{B'} A'\dar["\pi_B"]\\
B\rar[equal] & B
\etk\] and then consider the square shown above right, where $\varphi$ is the map induced by the universal property of the pullback. Note that $\pi_B$ is an acyclic fibration, since pullbacks preserve admissible epimorphisms and kernels of surjections. Then, if we let $s\colon B\to B\times_{B'} A'$ denote a lift for the square above right, we see that $\pi_{A'}s$ defines a lift for our original square. 

Dually, one shows that $\P_\Sp^\bot=\I_\Sp$.
\end{proof}

\begin{thm}
If the cotorsion pair $(\P,\I)$ is complete, then so is the cotorsion pair $(\P_\Sp,\I_\Sp)$.
\end{thm}
\begin{proof}
Let $X: C\xleftarrow{g} A\xrightarrow{f} B$ be an element in $\Span(\C)$. We show there exists  a resolution in $\Span(\C)$
\[\btk 0\rar & I\rar & P\rar & X\rar & 0\etk\] for some $I\in\I_\Sp$ and $P\in\P_\Sp$.

Since $(\P,\I)$ is complete, we can construct resolutions as pictured below left
\[\btk
0\rar & I_A\rar["\alpha_2"]
 & P_A\rar["\alpha_1"]
 & A
\dar["g"]\rar & 0\\
0\rar & I_C\rar["\gamma_2"'] & P_C\rar["\gamma_1"'] & C\rar & 0
\etk
\hspace{1cm}
\btk
0\rar[hookrightarrow]\dar & P_C\dar[twoheadrightarrow, "\gamma_1","\sim"']\\
P_A\rar["g\alpha_1"']\ar[ur, dashed, "g_1"] & C
\etk\] for some $P_A,  P_C\in\P$ and $I_A, I_C\in\I$. Since $\gamma_1$ is an acyclic fibration and $P_A\in\P$, we get a lift in the diagram above right, which induces a map on resolutions
\begin{equation}\label{resol1}
\btk
0\rar & I_A\rar["\alpha_2"]
\dar["g_2"] & P_A\rar["\alpha_1"]
\dar["g_1"] & A
\dar["g"]\rar & 0\\
0\rar & I_C\rar["\gamma_2"'] & P_C\rar["\gamma_1"'] & C\rar & 0
\etk
\end{equation}
Note, however, that there is no way to ensure that $g_2$ is an acyclic fibration. In order to fix this, we modify the given resolutions as follows.

Appealing to the completeness of $(\P,\I)$ once more, we get a resolution
\[\btk 0\rar & I_{I_C}\rar & P_{I_C}\rar["h"] & I_C\rar & 0\etk\]
form some $I_{I_C}\in\I$, $P_{I_C}\in\P$. Note that in this case we also have $P_{I_C}\in\I$, since $\I$ is closed under extensions.

Now replace the top half of our original resolution (\ref{resol1}) by
\begin{equation}\label{down}\btk
0\rar & I_A+P_{I_C}\rar["\alpha_2 + 1_Y"]\dar["g_2 +h"] & P_A+P_{I_C}\rar["{[\alpha_1,0]}"]\dar["g_1 +\gamma_2 h"] & A\dar["g"]\rar & 0\\
0\rar & I_C\rar["\gamma_2"'] & P_C\rar["\gamma_1"'] & C\rar & 0
\etk\end{equation} This is a commutative diagram with exact rows, and since $P_{I_C}\in\P\cap\I$, we have $P_A+P_{I_C}\in\P$ and $I_A +P_{I_C}\in\I$. Furthermore, the map $g_2 +h$ is an admissible epimorphism (because $h$ is), whose kernel is the pullback
\[\btk
I_A\times_{I_C} P_{I_C}\dar\rar & P_{I_C}\dar["h","\sim"', twoheadrightarrow]\\
I_A\rar["-g_2"'] & I_C
\etk\] which belongs to $\I$ by Lemma \ref{pullback}. Thus, $g_2 +h$ is an acyclic fibration.

By the same reasoning, there exists a resolution of $B$ and a commutative diagram
\[\btk
0\rar & I_B\rar["\beta_2"] & P_B\rar["\beta_1"] & B\rar & 0\\
0\rar & I_A+P_{I_C}\rar["\alpha_2 + 1_Y"']\uar["f_2"'] & P_A+P_{I_C}\rar["\alpha_1 +0"']\uar["f_1"'] & A\uar["f"']\rar & 0\\
 \etk\] Again, in order for this to be a part of the resolution we seek, we must modify it to compensate for the fact that $f_1$ is likely not a cofibration.

Using Theorem \ref{factorization}, we can factor $f_1$ as $\btk[column sep=small] P_A+P_{I_C}\rar[hookrightarrow,"i"] & \overline{P_B}\rar["p", twoheadrightarrow,"\sim"'] & P_B\etk$. Then $\overline{P_B}\in\P$, and we consider the diagram of exact rows
\begin{equation}\label{up}\btk
0\rar & \ker(\beta_1 p)\rar & \overline{P_B}\rar["\beta_1 p"] & B\rar & 0\\
0\rar & I_A+P_{I_C}\rar["\alpha_2 + 1_Y"']\uar & P_A+P_{I_C}\rar["\alpha_1 +0"']\uar["i"'] & A\uar["f"']\rar & 0\\
 \etk\end{equation} where $i$ is a cofibration, and $\ker(\beta_1 p)\in\I$ since both $\beta_1$ and $p$ are acyclic fibrations and these are closed under composition (Lemma \ref{compositions}).

 Pasting diagrams (\ref{down}) and (\ref{up}) together yields the desired resolution. Finally, dualizing the argument, one obtains a resolution
 \[\btk 0\rar & X\rar & I'\rar & P'\rar & 0\etk\]
  for some $I'\in \I_\Sp$, $P'\in\P_\Sp$.
\end{proof}


Theorem \ref{spans} shows that if $(\P,\I)$ is a complete cotorsion pair, then so is $(\P_\Sp, \I_\Sp)$, and thus Lemma \ref{compositions}, Proposition \ref{liftings} and Proposition \ref{factorization} also apply.

In this case, the classes of maps in $\Span(\C)$ that we get from $(\P_\Sp, \I_\Sp)$ are as follows:
\begin{itemize}
\item cofibrations: maps $(i,j,k)$ between objects of $\P_\Sp$ as below
\[\btk
C\dar["i"] & A\dar["j"]\lar\rar[hookrightarrow] & B\dar["k"]\\
C' & A'\lar\rar[hookrightarrow] & B'
\etk\] such that $i, j, k$ are admissible monomorphisms in $\C$, and \[\cok(i,j,k)=\cok i \leftarrow \cok j \to\cok k\in\P_\Sp.\] In other words, maps $(i,j,k)$ such that $i,j,k,$ and the induced map $\cok j\to\cok k$ are cofibrations in $\P$.
\item acyclic fibrations: maps $(o,p,q)$ between objects of $\Span(\C)$ as below
\[\btk
C\dar["o"] & A\dar["p"]\lar\rar & B\dar["q"]\\
C' & A'\lar\rar & B'
\etk\] such that $o, p, q$ are admissible epimorphisms in $\C$, and \[\ker(o,p,q)=\ker o \leftarrow\ker p \to \ker q \in\I_\Sp.\] More concretely, maps $(o,p,q)$ such that $o,p,q,$ and $\ker p\to\ker o$ are acyclic fibrations in $\C$.
\end{itemize}

\end{document}